\documentclass[11pt,a4paper]{article}

\usepackage[utf8]{inputenc}
\usepackage[T1]{fontenc}
\usepackage{lmodern}
\usepackage{amsmath,amssymb,amsthm}
\usepackage{mathrsfs}
\usepackage{mathtools}
\usepackage{bm}
\usepackage{graphicx}
\usepackage{hyperref}
\hypersetup{hidelinks}
\usepackage{enumitem}
\usepackage{geometry}
\usepackage{tikz}
\usetikzlibrary{arrows.meta}
\usepackage{float}
\usepackage{pgfplots}
\pgfplotsset{compat=1.17}

\newtheorem{theorem}{Theorem}[section]
\newtheorem{proposition}[theorem]{Proposition}

\newtheorem{corollary}[theorem]{Corollary}

\newtheorem{remark}[theorem]{Remark}

\title{Inverse Problems for the Return Map in the Class \(\mathcal{O}_C\):\\
Reconstruction and Identifiability}

\author{
	M. Barkatou\thanks{ISTM Laboratory, Chouaib Doukkali University, Morocco. \texttt{barkatou.m@ucd.ac.ma}}\\
	M. El Morsalani\thanks{QWave Consult, Germany. \texttt{Mohamed.elmorsalani@qwave-consult.eu}}
}

\date{}

\begin{document}

\maketitle

\begin{abstract}
	We study the inverse problem associated with the return map in the
	class \(\mathcal O_C\).  The convex core \(C\) is assumed known, while
	the return map
	\[
	F:\partial C\to\partial C
	\]
	is observed.  Decomposing the return displacement into tangential and
	normal components \(u_F\) and \(q_F\), we derive the exact nonlinear
	first-order reconstruction system
	\[
	(d-q_F)\nabla_M d=(I-dS)u_F
	\]
	for the unknown thickness function \(d\).  This yields one-point
	uniqueness: on a connected core boundary, an admissible thickness
	generating \(F\) is uniquely determined once its value is known at a
	single point.
	
	At a fixed point \(c_*\), the linearized return map satisfies
	\[
	M_*:=DF(c_*)-I
	=
	d_*(I-d_*S_*)^{-1}H_*,
	\]
	where \(H_*\) is the Hessian of the thickness.  Introducing
	\[
	Q_*:=S_*M_*-M_*^*S_*,
	\]
	we show that, when \(Q_*\neq0\), the thickness value \(d_*\), the
	Hessian \(H_*\), and the outer shape operator are uniquely recoverable
	from the observed linearization and the known geometry of \(C\).
	Combined with global propagation of the inverse system, a single such
	noncommuting fixed point determines the complete thickness function
	and hence the outer boundary.
	
	When \(Q_*=0\), the fixed-point linearization admits a one-parameter
	algebraic ambiguity, but this does not imply non-identifiability of
	the full return map.  The inverse system defines a nonlinear
	Ehresmann-type transport of anchor values along paths on
	\(\partial C\).  For closed paths, the associated nonlinear holonomy
	maps provide a second identifiability mechanism: a nontrivial
	linearized holonomy multiplier locally isolates the admissible
	anchor, while a monotonicity condition on one loop gives global anchor
	uniqueness.
	
	We also prove a conditional realizability result.  Any positive
	solution of the inverse system with \(d-q_F>0\) reconstructs an exact
	normal-incidence relation with the observed endpoint.  Under
	embeddedness, enclosure, first-hit transversality, and the remaining
	admissibility conditions of \(\mathcal O_C\), the reconstructed domain
	has return map \(F\).
	
	Finally, we identify dynamical and topological quantities recoverable
	directly from the return data, including the Morse index, fixed-point
	indices, Morse counting constraints, and the degree-one condition.
	Perturbation estimates show that the conditioning of the local
	thickness reconstruction deteriorates like
	\(1/\|Q_*\|_{\mathrm{HS}}\), and an explicit family realizes this
	blow-up.  Thus the commuting regime is a genuine singular limit of
	the local inverse problem, while nonlinear global transport may still
	retain identifying information.
\end{abstract}



\tableofcontents
\section{Introduction}

Let \(C\subset\mathbb R^N\) be a convex body and let
\(\Omega\in\mathcal O_C\) be an admissible surrounding domain.  The
outer boundary is represented relative to \(M:=\partial C\) by a
positive thickness function
\[
d:M\longrightarrow(0,\infty),
\qquad
\Phi(c)=c+d(c)\nu(c),
\]
where \(\nu\) denotes the outward unit normal to \(M\).  The normal
round trip from \(c\in M\) to \(\Phi(c)\in\partial\Omega\) and back to
\(M\) defines the return map
\[
F:M\longrightarrow M.
\]

The forward correspondence
\[
(C,\Omega)\longmapsto d\longmapsto F
\]
was developed in \cite{BEM2026ReturnMap,EMB2026Global,EMB2026Morse}.
The first paper established the exact geometric structure of the
return map and the fixed-point characterization
\[
F(c)=c
\quad\Longleftrightarrow\quad
\nabla_Md(c)=0.
\]
The second identified regimes in which the thickness is a strict
Lyapunov function for the return dynamics, while the third related
the linearized map to the Hessian of \(d\), the curvature of the two
boundaries, and the Morse topology of the thickness landscape.

The present paper studies the inverse direction.  The convex core
\(C\), and hence its normal field and shape operator, are assumed
known, while the return map \(F\) is observed.  The basic questions
are whether the thickness function and the outer boundary can be
reconstructed from \(F\), which geometric quantities remain
observable when full reconstruction fails, which compatibility
conditions characterize admissible return data, and how the inverse
problem behaves under perturbation.

\subsection{Inverse problem and main results}
\label{subsec:intro-main-results}

Write \(M=\partial C\), and decompose the observed displacement as
\[
u_F(c)
=
P_{T_cM}\bigl(F(c)-c\bigr),
\qquad
q_F(c)
=
\langle F(c)-c,\nu(c)\rangle.
\]
The exact return identity yields the nonlinear first-order inverse
system
\begin{equation}
(d-q_F)\nabla_Md=(I-dS)u_F.
\label{eq:intro-inverse-system}
\end{equation}
Thus the complete return displacement determines a differential
system for the unknown thickness rather than merely an asymptotic
gradient direction.

A first consequence is one-point uniqueness.  If two admissible
thickness functions generate the same return map and agree at one
point of the connected boundary \(M\), then they agree everywhere.
Exact non-uniqueness is therefore reduced to the determination of a
single scalar anchor value.

At a fixed point \(c_*\in\operatorname{Fix}(F)\), set
\[
M_*:=DF(c_*)-I,
\qquad
S_*:=S_{c_*}.
\]
The exact linearization has the structured factorization
\[
M_*
=
d_*(I-d_*S_*)^{-1}H_*,
\qquad
H_*=\operatorname{Hess}_Md(c_*),
\]
and consequently
\[
M_*-M_*^*
=
d_*
\bigl(S_*M_*-M_*^*S_*\bigr).
\]
Introducing
\[
Q_*:=S_*M_*-M_*^*S_*,
\]
one obtains, whenever \(Q_*\neq0\),
\[
d_*
=
\frac{
	\langle M_*-M_*^*,Q_*\rangle_{\mathrm{HS}}
}{
	\|Q_*\|_{\mathrm{HS}}^2
}.
\]
The Hessian \(H_*\) and the outer shape operator
\(S_*^\Omega\) then follow explicitly.  Combined with the global
uniqueness of \eqref{eq:intro-inverse-system}, this gives the first
global identifiability mechanism of the paper:

\begin{quote}
	If the return map has one fixed point with \(Q_*\neq0\), then the
	exact return map \(F\), together with the known core \(C\), uniquely
	determines the full thickness function \(d\) and hence
	\(\partial\Omega\).
\end{quote}

The complementary case \(Q_*=0\) has a different character.  For
exact data it is equivalent to self-adjointness of \(M_*\) together
with commutation of \(M_*\) and \(S_*\).  The derivative
\(DF(c_*)\) then admits a one-parameter algebraic factorization and
does not by itself determine the thickness value.  This includes
umbilic cores and, in particular, spherical geometry.

The absence of local noncommutativity, however, does not imply
non-identifiability of the full return map.  Equation
\eqref{eq:intro-inverse-system} defines nonlinear transport of trial
anchor values along paths in \(M\).  For closed paths this transport
gives nonlinear holonomy maps.  The true anchor is a common fixed
point of these maps, and a nontrivial linearized holonomy multiplier
isolates it locally.  Under a suitable monotonicity condition on one
loop holonomy, the admissible anchor is globally unique.  Thus local
noncommutativity and nonlinear loop transport provide two distinct
mechanisms for identifiability.

The inverse system also yields a partial realizability result.  If a
positive function \(d\) satisfies
\eqref{eq:intro-inverse-system} with
\[
d-q_F>0,
\]
then the reconstructed normal graph already satisfies the exact
normal-incidence relation with the observed endpoint \(F(c)\).  If,
in addition, the graph is an admissible embedded outer boundary and
the corresponding inward normal ray meets \(M\) first at \(F(c)\),
then the reconstructed domain has return map \(F\).  This reduces the
remaining range problem to global solvability of the inverse system
together with geometric admissibility of the reconstructed graph.

Several quantities are observable even when the anchor itself is not
recovered from a fixed-point linearization.  At every nondegenerate
fixed point,
\[
\operatorname{ind}_M(c_*)
=
n_-\bigl(DF(c_*)-I\bigr).
\]
The fixed-point index, Morse counting relations, Euler-characteristic
balance, and the degree-one condition therefore become directly
testable constraints on candidate return data.

Finally, the same quantity \(Q_*\) that governs local
identifiability also controls conditioning.  Perturbation estimates
show that the sensitivity of the reconstructed anchor deteriorates
as
\[
\|Q_*\|_{\mathrm{HS}}\longrightarrow0,
\]
and an explicit perturbation family realizes the
\(1/\|Q_*\|_{\mathrm{HS}}\) blow-up.  The commuting configuration is
therefore a genuine singular limit of the local inverse problem.

\begin{table}[H]
	\centering
	\renewcommand{\arraystretch}{1.15}
	\begin{tabular}{ll}
		\hline
		Symbol & Meaning \\
		\hline
		\(M=\partial C\)
		&
		known boundary of the convex core
		\\
		
		\(\nu\)
		&
		outward unit normal to \(M\)
		\\
		
		\(S\)
		&
		shape operator of \(M\), with \(S=-D\nu|_{TM}\)
		\\
		
		\(d\)
		&
		thickness function defining
		\(\Phi(c)=c+d(c)\nu(c)\)
		\\
		
		\(u_F,\;q_F\)
		&
		tangential and normal components of \(F(c)-c\)
		\\
		
		\(c_*\)
		&
		fixed point of \(F\), equivalently a critical point of \(d\)
		\\
		
		\(M_*:=DF(c_*)-I\)
		&
		linearized return displacement at \(c_*\)
		\\
		
		\(H_*:=\operatorname{Hess}_M d(c_*)\)
		&
		intrinsic Hessian of the thickness at \(c_*\)
		\\
		
		\(S_*\)
		&
		shape operator \(S\) evaluated at \(c_*\)
		\\
		
		\(S_*^\Omega\)
		&
		shape operator of the outer boundary at \(\Phi(c_*)\)
		\\
		
		\(Q_*:=S_*M_*-M_*^*S_*\)
		&
		observable noncommuting quantity governing local identifiability
		\\
		
		\(\mathcal T_\gamma\)
		&
		nonlinear thickness transport along a path \(\gamma\)
		\\
		
		\(\mathcal A_F(c_0)\)
		&
		admissible anchor set at the base point \(c_0\)
		\\
		\hline
	\end{tabular}
	\caption{Notation used throughout the inverse analysis.}
	\label{tab:inverse-notation}
\end{table}

Throughout the paper, a subscript \(*\) denotes evaluation at a fixed
point \(c_*\), whereas a superscript \(*\) denotes the adjoint of a
linear operator.
\subsection{Relation with geometric inverse and rigidity problems}

The problem belongs to a broader class of geometric inverse questions
in which a geometry generates boundary dynamics and one asks to what
extent the induced dynamics determines the underlying geometry.  A
closely related example is the inverse problem for convex billiards,
where the shape of the billiard table is studied through dynamical,
spectral, or length data; see, for example, the survey of Kaloshin
and Sorrentino \cite{KaloshinSorrentino2022}.  The return mechanism
considered here is different from specular billiard reflection, but
the inverse question is structurally similar: a geometric domain is
observed through a boundary correspondence generated by it.

A second nearby line is boundary and scattering rigidity.  Recent
results recover geometric or dynamical structures from boundary
action or scattering relations in magnetic--potential systems,
Lorentzian geometry, and Hamiltonian systems
\cite{MunozThon2024,Stefanov2025,EptaminitakisStefanov2026}.
Related inverse-scattering results recover Riemannian metrics or sound
speeds from restricted scattering information
\cite{OksanenRakeshSalo2026}.  In these settings the data typically
contain incoming and outgoing directions, travel times, or
phase-space information.

Here the observable is instead a discrete boundary map
\[
F:M\longrightarrow M
\]
generated by the normal geometry of two nested boundaries.  The
inverse mechanism is correspondingly different.  Reconstruction
reduces to a nonlinear first-order system on the known boundary,
together with a structured factorization at fixed points and global
compatibility of nonlinear path transport.  In particular,
anisotropy and noncommutativity can provide local identifiability,
while nonlinear loop transport can retain identifying information
even when the local fixed-point mechanism becomes commuting.

\subsection{Motivation and possible applications}

The return-map model provides a setting in which geometry, discrete
dynamics, and inverse reconstruction can be studied within the same
exact framework.  The fixed points of \(F\) detect critical
thickness, the derivative \(DF\) contains second-order geometric
information, and the complete displacement determines a nonlinear
transport equation for the unknown boundary separation.  The model
therefore separates information that is locally observable,
information that requires global compatibility, and information that
is lost or poorly conditioned because of symmetry.

The framework may also be viewed as an idealized model for geometric
probing of an unknown interface surrounding a known reference body.
The measured data are boundary input--output locations, possibly
supplemented by derivatives of the input--output map.  The
reconstruction results are therefore relevant as a geometric model
for shape and interface recovery, surface metrology, geometric
sensing, and system identification from boundary-to-boundary
observations.

The stability analysis is particularly relevant when the return map
is estimated from finite or noisy measurements.  It distinguishes
between configurations that are geometrically identifiable and those
that become increasingly difficult to reconstruct near the commuting
regime.

More broadly, the model provides a finite-dimensional geometric
counterpart to scattering and rigidity problems.  It isolates the
roles of exact boundary correspondence, local linearization, global
transport, symmetry, admissibility, and measurement error without
introducing an underlying PDE.

\subsection{Organization of the paper}

Section~2 recalls the forward geometric identities needed for the
inverse analysis.  Section~3 derives the exact nonlinear
reconstruction system, its transport formulation, one-point
uniqueness, and basic compatibility conditions.  Section~4 studies
local reconstruction at fixed points and separates the noncommuting
and commuting regimes.  Section~5 records dynamical and topological
quantities that are directly observable from the return map.

Section~6 develops the global inverse theory: propagation through the
fixed-point set, uniqueness from one anchor, global identifiability
from a noncommuting fixed point, path independence, conditional
realizability of reconstructed data, and nonlinear holonomy criteria
for anchor identifiability, including the all-commuting regime.
Section~7 studies perturbations, conditioning, and consistency
residuals.  Section~8 gives explicit and computed examples illustrating local
identifiability, conditioning near the commuting regime, and global
reconstruction by nonlinear loop transport.  Section~9 discusses the remaining problems, including
residual all-commuting ambiguity, intrinsic realizability, lower
regularity, degenerate critical sets, and global stability.
\section{Geometric Framework and Exact Forward Identities}
\label{sec:geometric-framework}

We recall only the geometric identities needed for the inverse
analysis. The construction of the class \(\mathcal O_C\), the
well-posedness of the return map, and the corresponding regularity
results are established in
\cite{Barkatou2002,BEM2026ReturnMap}.

Throughout the paper, \(C\subset\mathbb R^N\) is a compact convex body
with nonempty interior and sufficiently smooth boundary. We write
\[
M:=\partial C,
\qquad
m:=\dim M=N-1,
\]
and denote by
\[
\nu:M\longrightarrow \mathbb S^{N-1}
\]
the outward unit normal.

The shape operator of \(M\) is defined by
\begin{equation}
S_c
=
-D\nu(c)|_{T_cM}.
\label{eq:shape-operator}
\end{equation}
With this convention, if
\(\kappa_1^C(c),\ldots,\kappa_m^C(c)\ge0\)
are the usual principal curvatures of the convex boundary, then the
eigenvalues of \(S_c\) are
\[
-\kappa_1^C(c),\ldots,-\kappa_m^C(c).
\]

Let \(\Omega\in\mathcal O_C\) be described by the positive thickness
function
\[
d:M\longrightarrow(0,\infty)
\]
through the normal graph
\begin{equation}
\Phi(c)
=
c+d(c)\nu(c).
\label{eq:radial-map}
\end{equation}
Thus
\[
\partial\Omega
=
\Phi(M)
\]
in the regular setting considered below.

We assume the regularity required by the identities being used; in
particular, the local linearization and curvature formulas will be
applied under
\[
\partial C\in C^3,
\qquad
d\in C^2(M).
\]

\subsection{Normal graph geometry}
\label{subsec:normal-graph-geometry}

Set
\begin{equation}
g_c
:=
\nabla_M d(c),
\qquad
L_c
:=
I-d(c)S_c.
\label{eq:g-L-definition}
\end{equation}
Since \(C\) is convex,
\[
\sigma(L_c)
=
\left\{
1+d(c)\kappa_i^C(c)
:
i=1,\ldots,m
\right\},
\]
and therefore
\begin{equation}
L_c>0.
\label{eq:L-positive}
\end{equation}

Define
\begin{equation}
w_c
:=
L_c^{-1}g_c.
\label{eq:w-definition}
\end{equation}
Differentiating \(\Phi\) gives, for \(v\in T_cM\),
\begin{equation}
D\Phi(c)[v]
=
L_cv
+
\langle g_c,v\rangle\nu(c).
\label{eq:D-Phi}
\end{equation}

The outward unit normal to \(\partial\Omega\) at
\(x=\Phi(c)\) is
\begin{equation}
N_\Omega(\Phi(c))
=
\frac{\nu(c)-w_c}
{\sqrt{1+|w_c|^2}},
\label{eq:outer-normal}
\end{equation}
and the inward unit normal is
\begin{equation}
n_\Omega(\Phi(c))
=
\frac{-\nu(c)+w_c}
{\sqrt{1+|w_c|^2}}.
\label{eq:inward-normal}
\end{equation}
These identities are exact; see
\cite{BEM2026ReturnMap}.

\subsection{Exact return decomposition}
\label{subsec:exact-return-decomposition}

Let \(t(c)>0\) denote the distance travelled from
\(\Phi(c)\) along the inward normal until the first return to
\(\partial C\). Then
\begin{equation}
F(c)
=
\Phi(c)
+
t(c)n_\Omega(\Phi(c)).
\label{eq:return-definition}
\end{equation}
Introduce the scaled return distance
\begin{equation}
\sigma(c)
:=
\frac{t(c)}
{\sqrt{1+|w_c|^2}}.
\label{eq:sigma-definition}
\end{equation}
Substituting \eqref{eq:radial-map} and
\eqref{eq:inward-normal} into
\eqref{eq:return-definition} yields
\begin{equation}
\boxed{
	F(c)-c
	=
	\sigma(c)w_c
	+
	\bigl(d(c)-\sigma(c)\bigr)\nu(c).
}
\label{eq:exact-return-decomposition}
\end{equation}

Consequently, the tangential and normal components of the return
displacement are
\begin{align}
P_{T_cM}\bigl(F(c)-c\bigr)
&=
\sigma(c)w_c,
\label{eq:exact-tangent-component}
\\
\left\langle F(c)-c,\nu(c)\right\rangle
&=
d(c)-\sigma(c).
\label{eq:exact-normal-component}
\end{align}

Equation~\eqref{eq:exact-return-decomposition} is the principal
forward identity used in the exact inverse problem.

Locally, write \(M\) near \(c\) as a normal graph
\[
c+u+h_c(u)\nu(c),
\qquad
u\in T_cM,
\]
with
\[
h_c(0)=0,
\qquad
Dh_c(0)=0.
\]
Since \(F(c)\in M\), the exact decomposition also gives the closure
relation
\begin{equation}
d(c)-\sigma(c)
=
h_c\bigl(\sigma(c)w_c\bigr).
\label{eq:closure-equation}
\end{equation}
This identity will later provide a compatibility condition for
candidate return data.

The local expansion obtained in
\cite{BEM2026ReturnMap} is
\begin{equation}
F(c)-c
=
B_c\,g_c
+
O(|g_c|^2),
\label{eq:first-order-return}
\end{equation}
where
\begin{equation}
B_c
:=
d(c)L_c^{-1}
=
d(c)\bigl(I-d(c)S_c\bigr)^{-1}.
\label{eq:B-definition}
\end{equation}
The operator \(B_c\) is self-adjoint and positive definite, with
principal eigenvalues
\begin{equation}
\beta_i(c)
=
\frac{d(c)}
{1+d(c)\kappa_i^C(c)}
>0.
\label{eq:B-eigenvalues}
\end{equation}

We emphasize that
\eqref{eq:exact-return-decomposition} is the exact geometric identity;
\eqref{eq:first-order-return} will be used only when a local
asymptotic statement is required.

\subsection{Fixed points and exact linearization}
\label{subsec:fixed-points-linearization}

The fixed points of the return map are characterized exactly by
\begin{equation}
\boxed{
	F(c)=c
	\quad\Longleftrightarrow\quad
	\nabla_M d(c)=0.
}
\label{eq:fixed-critical}
\end{equation}
Thus the fixed-point set of \(F\) coincides with the critical set of
the thickness landscape.

Let \(c_*\in M\) be such a point and write
\[
d_*=d(c_*),
\qquad
S_*=S_{c_*},
\qquad
L_*=I-d_*S_*,
\]
and
\begin{equation}
H_*
:=
\operatorname{Hess}_M d(c_*).
\label{eq:H-definition}
\end{equation}
Then
\begin{equation}
B_*
=
d_*L_*^{-1},
\label{eq:B-star}
\end{equation}
and the exact linearization of the return map is
\begin{equation}
\boxed{
	DF(c_*)
	=
	I+B_*H_*.
}
\label{eq:exact-linearization}
\end{equation}

It is convenient to introduce
\begin{equation}
K_*
:=
B_*^{1/2}H_*B_*^{1/2}.
\label{eq:K-definition}
\end{equation}
Although \(B_*H_*\) need not be self-adjoint,
\begin{equation}
B_*H_*
=
B_*^{1/2}K_*B_*^{-1/2},
\label{eq:BH-similarity}
\end{equation}
so \(B_*H_*\) and \(K_*\) have the same real spectrum.
Moreover,
\begin{equation}
\operatorname{inertia}(K_*)
=
\operatorname{inertia}(H_*),
\label{eq:inertia-preservation}
\end{equation}
because \(K_*\) is positively congruent to \(H_*\).
These observations will be used in the local inverse problem.

\subsection{Curvature representation at a critical point}
\label{subsec:critical-curvature-identity}

At a critical point \(c_*\), one has \(w_{c_*}=0\), and therefore
\[
N_\Omega(\Phi(c_*))
=
\nu(c_*).
\]
The tangent spaces of the two boundaries can consequently be
identified through \(D\Phi(c_*)=L_*\).

Let
\[
S_*^\Omega
\]
denote the shape operator of \(\partial\Omega\) at
\[
x_*=\Phi(c_*),
\]
using the same outward-normal convention as for \(S_*\).
The curvature calculation of
\cite{EMB2026Morse} gives
\begin{equation}
\boxed{
	H_*
	=
	L_*S_*^\Omega L_*
	-
	L_*S_*.
}
\label{eq:H-curvature}
\end{equation}
Combining this with
\eqref{eq:K-definition} yields
\begin{equation}
\boxed{
	K_*
	=
	d_*
	\left(
	L_*^{1/2}S_*^\Omega L_*^{1/2}
	-
	S_*
	\right).
}
\label{eq:K-curvature}
\end{equation}

No simultaneous diagonalization of \(S_*\) and
\(S_*^\Omega\) is required in
\eqref{eq:H-curvature}--\eqref{eq:K-curvature}.
This point will be essential in the inverse analysis: the
noncommuting case contains information that is lost in certain
aligned configurations.

\subsection{Global ascent regime}
\label{subsec:global-ascent-forward}

The exact inverse identities above do not require a global convergence
hypothesis. Some later dynamical interpretations, however, will use the global
dynamical regime established in \cite{EMB2026Global}.

When the corresponding strict-ascent assumptions hold, there exists
\(\gamma>0\) such that
\begin{equation}
d(F(c))-d(c)
\ge
\gamma
|\nabla_M d(c)|^2
\label{eq:strict-ascent}
\end{equation}
away from the critical set. Consequently,
\[
d(F(c))>d(c)
\qquad
\text{for }
c\notin\operatorname{Crit}(d),
\]
and
\[
d_{\max}-d
\]
is a strict Lyapunov function relative to the fixed-point set.

Under the additional isolation hypotheses stated in
\cite{EMB2026Global}, every orbit converges to a critical point of
\(d\), and nontrivial periodic or recurrent orbits are excluded.

We shall invoke these global consequences only where explicitly
stated. The local and exact inverse results of the next sections do
not depend on them.

\section{Exact Reconstruction Away from the Fixed-Point Set}
\label{sec:exact-inverse}

We now regard \(C\) and the return map \(F\) as known.  The observable
fields \(u_F\) and \(q_F\) are those defined in
Section~\ref{subsec:intro-main-results}.  The exact return
decomposition of Section~\ref{subsec:exact-return-decomposition}
converts the inverse problem into a first-order nonlinear system for
the unknown thickness \(d\).

\subsection{The exact inverse system}
\label{subsec:inverse-system}

Let
\[
\mathcal U
:=
M\setminus\operatorname{Fix}(F).
\]
By \eqref{eq:fixed-critical},
\[
\mathcal U
=
M\setminus\operatorname{Crit}(d).
\]
Hence
\[
\nabla_M d(c)\neq0
\qquad
(c\in\mathcal U).
\]

From \eqref{eq:exact-tangent-component} and
\eqref{eq:exact-normal-component},
\[
u_F=\sigma(I-dS)^{-1}\nabla_M d,
\qquad
q_F=d-\sigma.
\]
Since
\[
\sigma=d-q_F>0,
\]
we obtain the exact equation
\begin{equation}
\boxed{
	(d-q_F)\nabla_M d
	=
	(I-dS)u_F.
}
\label{eq:exact-inverse-system}
\end{equation}

In particular,
\begin{equation}
u_F(c)\neq0,
\qquad
c\in\mathcal U.
\label{eq:uF-nonzero}
\end{equation}

It is convenient to introduce, for \(z>q_F(c)\), the one-form
\begin{equation}
\omega_F(c,z)[v]
:=
\frac{
	\left\langle
	(I-zS_c)u_F(c),v
	\right\rangle
}{
	z-q_F(c)
},
\qquad
v\in T_cM.
\label{eq:inverse-one-form}
\end{equation}
Then \eqref{eq:exact-inverse-system} is equivalent to
\begin{equation}
\boxed{
	d_Md
	=
	\omega_F(\,\cdot\,,d).
}
\label{eq:inverse-pfaffian}
\end{equation}

Thus the exact inverse problem is a nonlinear Pfaffian system whose
coefficients are determined entirely by \(F\) and the known geometry
of \(C\).

In a principal frame
\[
S_ce_i=-\kappa_i^C(c)e_i,
\qquad
u_F=\sum_{i=1}^m u_i e_i,
\]
equation~\eqref{eq:exact-inverse-system} becomes
\begin{equation}
\boxed{
	(d-q_F)e_i(d)
	=
	\bigl(1+d\kappa_i^C\bigr)u_i,
	\qquad
	i=1,\ldots,m.
}
\label{eq:inverse-principal-frame}
\end{equation}

Unlike the first-order approximation
\eqref{eq:first-order-return},
equations
\eqref{eq:exact-inverse-system}--\eqref{eq:inverse-principal-frame}
contain no asymptotic remainder.

\subsection{Transport along observable directions}
\label{subsec:inverse-transport}

Let
\[
\gamma:[0,T]\longrightarrow\mathcal U
\]
be a \(C^1\) curve and set
\[
y(t):=d(\gamma(t)).
\]
Evaluating \eqref{eq:inverse-pfaffian} on
\(\dot\gamma(t)\) gives
\begin{equation}
\boxed{
	y'(t)
	=
	\frac{
		\left\langle
		\bigl(I-y(t)S_{\gamma(t)}\bigr)
		u_F(\gamma(t)),
		\dot\gamma(t)
		\right\rangle
	}{
		y(t)-q_F(\gamma(t))
	}.
}
\label{eq:pathwise-reconstruction}
\end{equation}

Thus reconstruction along any prescribed path reduces to a scalar
first-order ordinary differential equation.

A particularly natural choice is an integral curve of the observable
tangent field:
\begin{equation}
\dot\gamma(t)
=
u_F(\gamma(t)).
\label{eq:uF-integral-curve}
\end{equation}
Then
\begin{equation}
\bigl(y-q_F\bigr)y'
=
|u_F|^2
-
y\langle Su_F,u_F\rangle.
\label{eq:uF-transport}
\end{equation}

Since \(C\) is convex,
\[
\langle S_cv,v\rangle\le0
\qquad
(v\in T_cM),
\]
and therefore
\begin{equation}
|u_F|^2
-
d\langle Su_F,u_F\rangle
=
\langle (I-dS)u_F,u_F\rangle
>0
\label{eq:transport-positive}
\end{equation}
on \(\mathcal U\). Because \(d-q_F=\sigma>0\), we obtain
\begin{equation}
\boxed{
	u_F[d]>0
	\qquad
	\text{on }\mathcal U.
}
\label{eq:uF-thickness-monotonicity}
\end{equation}

Hence the orientation of the observable tangent field is fixed by the
thickness landscape: its integral curves cross the level sets of \(d\)
strictly in the direction of increasing thickness.

\subsection{One-point uniqueness}
\label{subsec:one-point-uniqueness}

Equation~\eqref{eq:pathwise-reconstruction} has an immediate
identifiability consequence.

\begin{theorem}[One-point uniqueness]
	\label{thm:one-point-uniqueness}
	Let \(U\subset M\) be path connected, and let
	\[
	d_1,d_2\in C^1(U)
	\]
	be two positive solutions of the same inverse system
	\eqref{eq:exact-inverse-system} determined by \(F\) and \(C\).
	Assume
	\begin{equation}
	d_j(c)-q_F(c)>0,
	\qquad
	c\in U,
	\quad
	j=1,2.
	\label{eq:sigma-positive-two-solutions}
	\end{equation}
	If there exists \(c_0\in U\) such that
	\[
	d_1(c_0)=d_2(c_0),
	\]
	then
	\[
	\boxed{
		d_1=d_2
		\quad\text{on }U.
	}
	\]
\end{theorem}

\begin{proof}
	Fix \(c\in U\) and choose a piecewise \(C^1\) path
	\[
	\gamma:[0,1]\to U,
	\qquad
	\gamma(0)=c_0,
	\qquad
	\gamma(1)=c.
	\]
	Set
	\[
	y_j(t)=d_j(\gamma(t)).
	\]
	By \eqref{eq:pathwise-reconstruction}, both functions solve
	\begin{equation}
	y'
	=
	G_\gamma(t,y),
	\label{eq:pathwise-ode}
	\end{equation}
	where
	\[
	G_\gamma(t,y)
	=
	\frac{
		\left\langle
		(I-yS_{\gamma(t)})u_F(\gamma(t)),
		\dot\gamma(t)
		\right\rangle
	}{
		y-q_F(\gamma(t))
	}.
	\]
	On the region \(y>q_F(\gamma(t))\), the right-hand side is locally
	Lipschitz in \(y\).  Since
	\[
	y_1(0)=y_2(0),
	\]
	uniqueness for the scalar ODE gives
	\[
	y_1(t)=y_2(t)
	\]
	along the path. In particular,
	\[
	d_1(c)=d_2(c).
	\]
	Since \(c\in U\) was arbitrary, the result follows.
\end{proof}

\begin{corollary}[Reduction of exact non-uniqueness]
	\label{cor:anchor-value-identifiability}
	Let \(d\) and \(\widetilde d\) be two admissible thickness functions
	for the same convex core \(C\) and suppose that they generate the same
	return map \(F\). On every path-connected region on which
	\[
	d-q_F>0,
	\qquad
	\widetilde d-q_F>0,
	\]
	equality of the two thickness functions at one point implies
	\[
	d=\widetilde d
	\]
	throughout that region.
\end{corollary}

Thus any remaining exact ambiguity cannot consist of an arbitrary
pointwise modification of the thickness function. It is reduced to
the choice of an initial thickness value on each connected
reconstruction region.

\subsection{Compatibility of candidate return data}
\label{subsec:return-data-compatibility}

Not every observed boundary correspondence is compatible with the
return construction.

First, the normal-graph closure relation
\begin{equation}
\boxed{
	q_F(c)
	=
	h_c\bigl(u_F(c)\bigr).
}
\label{eq:observable-closure-condition}
\end{equation}
expresses endpoint consistency with the known core boundary.
If \(F:M\to M\) is assumed a priori, this identity is automatic,
since it merely states that the observed endpoint \(F(c)\) lies on
\(M\).  It becomes a nontrivial consistency check only when the
ambient displacement is supplied independently or is affected by
measurement error.

The genuinely inverse compatibility conditions arise from the
existence of a thickness function satisfying the nonlinear system.
If a \(C^2\) function \(d\) solves \eqref{eq:inverse-pfaffian}, then the one-form
\[
\alpha_d
:=
\omega_F(\,\cdot\,,d(\,\cdot\,))
\]
is exact:
\[
\alpha_d=d_Md.
\]
Hence it necessarily satisfies
\begin{equation}
\boxed{
	d_M\alpha_d=0.
}
\label{eq:inverse-integrability}
\end{equation}

Equation~\eqref{eq:inverse-integrability} is the integrability
condition for the nonlinear first-order system.  It does not by
itself produce a solution, but it provides a necessary consistency
test for a candidate reconstruction.

At a fixed point \(c_*\), one has
\[
u_F(c_*)=0,
\qquad
q_F(c_*)=0,
\]
and \eqref{eq:exact-inverse-system} reduces to
\[
\nabla_Md(c_*)=0.
\]
The return displacement therefore determines criticality but not the
value \(d(c_*)\). Recovering this missing scalar datum from
\(DF(c_*)\) is the subject of the next section.

\section{Local Reconstruction at Fixed Points}
\label{sec:local-inverse}

Let \(c_*\in\operatorname{Fix}(F)\). By
\eqref{eq:fixed-critical},
\[
\nabla_M d(c_*)=0.
\]
Set
\[
d_*:=d(c_*),
\qquad
S_*:=S_{c_*},
\qquad
H_*:=\operatorname{Hess}_M d(c_*),
\]
and define the observable linearized displacement
\begin{equation}
M_*:=DF(c_*)-I.
\label{eq:M-star-definition}
\end{equation}

From \eqref{eq:exact-linearization},
\begin{equation}
M_*
=
d_*(I-d_*S_*)^{-1}H_*.
\label{eq:structured-factorization}
\end{equation}

The inverse problem at \(c_*\) is to recover
\[
d_*,
\qquad
H_*,
\qquad
S_*^\Omega
\]
from the known operators \(M_*\) and \(S_*\).

\subsection{The structured self-adjointness identity}
\label{subsec:structured-selfadjointness}

Let
\[
L_*:=I-d_*S_*.
\]
Equation~\eqref{eq:structured-factorization} gives
\begin{equation}
L_*M_*=d_*H_*.
\label{eq:LM-dH}
\end{equation}
Since \(H_*\) is self-adjoint,
\[
L_*M_*
=
M_*^*L_*.
\]
Substituting \(L_*=I-d_*S_*\) yields
\begin{equation}
\boxed{
	M_*-M_*^*
	=
	d_*
	\bigl(S_*M_*-M_*^*S_*\bigr).
}
\label{eq:local-identifiability-identity}
\end{equation}

Define
\begin{equation}
Q_*:=
S_*M_*-M_*^*S_*.
\label{eq:Q-star}
\end{equation}
Then
\begin{equation}
M_*-M_*^*
=
d_*Q_*.
\label{eq:skew-Q}
\end{equation}

Thus the skew-adjoint part of the observed linearization is not
arbitrary: it is constrained by the known core curvature and the
single scalar \(d_*\).

\subsection{Thickness recovery in the noncommuting regime}
\label{subsec:noncommuting-recovery}

\begin{theorem}[Local thickness reconstruction]
	\label{thm:local-thickness-reconstruction}
	Assume
	\begin{equation}
	Q_*\neq0.
	\label{eq:noncommuting-condition}
	\end{equation}
	Then the thickness value \(d_*\) is uniquely determined by
	\(DF(c_*)\) and the known shape operator \(S_*\). More precisely,
	\begin{equation}
	\boxed{
		d_*
		=
		\frac{
			\left\langle
			M_*-M_*^*,Q_*
			\right\rangle_{\mathrm{HS}}
		}{
			\|Q_*\|_{\mathrm{HS}}^2
		}.
	}
	\label{eq:d-star-reconstruction}
	\end{equation}
\end{theorem}

\begin{proof}
	By \eqref{eq:skew-Q},
	\[
	M_*-M_*^*=d_*Q_*.
	\]
	Taking the Hilbert--Schmidt inner product with \(Q_*\) gives
	\[
	\left\langle
	M_*-M_*^*,Q_*
	\right\rangle_{\mathrm{HS}}
	=
	d_*\|Q_*\|_{\mathrm{HS}}^2.
	\]
	Since \(Q_*\neq0\), division by
	\(\|Q_*\|_{\mathrm{HS}}^2\) yields
	\eqref{eq:d-star-reconstruction}.
\end{proof}

For exact return data, the right-hand side of
\eqref{eq:d-star-reconstruction} is necessarily positive.

The result can also be written in a principal frame of the convex
core. Let
\[
S_*e_i=-\kappa_i^C e_i
\]
and write
\[
m_{ij}:=\langle M_*e_j,e_i\rangle.
\]
Equation~\eqref{eq:local-identifiability-identity} gives
\begin{equation}
m_{ij}-m_{ji}
=
d_*
\left(
\kappa_j^C m_{ji}
-
\kappa_i^C m_{ij}
\right).
\label{eq:entrywise-identifiability}
\end{equation}
Hence any pair \(i\neq j\) for which
\[
\kappa_j^C m_{ji}
-
\kappa_i^C m_{ij}
\neq0
\]
already determines
\begin{equation}
d_*
=
\frac{m_{ij}-m_{ji}}
{\kappa_j^C m_{ji}-\kappa_i^C m_{ij}}.
\label{eq:entrywise-d}
\end{equation}

Equation~\eqref{eq:entrywise-d} makes explicit that local
identifiability is generated by the interaction between anisotropy of
the core and coupling in the observed linearization.

\subsection{Recovery of the Hessian and outer curvature}
\label{subsec:hessian-curvature-recovery}

Once \(d_*\) has been reconstructed, the remaining local geometric
quantities follow directly.

\begin{proposition}[Hessian reconstruction]
	\label{prop:hessian-reconstruction}
	Under the hypotheses of
	Theorem~\ref{thm:local-thickness-reconstruction},
	\begin{equation}
	\boxed{
		H_*
		=
		\frac{1}{d_*}
		(I-d_*S_*)M_*.
	}
	\label{eq:H-star-reconstruction}
	\end{equation}
	Thus the full intrinsic Hessian of the thickness function at \(c_*\)
	is determined by \(DF(c_*)\) and the known geometry of \(C\).
\end{proposition}

\begin{proof}
	This is the rearrangement of
	\eqref{eq:LM-dH}.
\end{proof}

The curvature identity
\eqref{eq:H-curvature}, established in
\cite{EMB2026Morse}, then gives the outer shape operator.

\begin{proposition}[Outer-curvature reconstruction]
	\label{prop:outer-curvature-reconstruction}
	Let
	\[
	x_*=\Phi(c_*).
	\]
	Under the same hypotheses,
	\begin{equation}
	\boxed{
		S_*^\Omega
		=
		\left(
		S_*+\frac{1}{d_*}M_*
		\right)
		(I-d_*S_*)^{-1}.
	}
	\label{eq:S-Omega-reconstruction}
	\end{equation}
	Hence the second fundamental form of \(\partial\Omega\) at \(x_*\)
	is locally identifiable from the return linearization.
\end{proposition}

\begin{proof}
	From \eqref{eq:H-curvature},
	\[
	H_*
	=
	L_*S_*^\Omega L_*
	-
	L_*S_*.
	\]
	Multiplying on the left by \(L_*^{-1}\) gives
	\[
	L_*^{-1}H_*
	=
	S_*^\Omega L_*-S_*.
	\]
	By \eqref{eq:structured-factorization},
	\[
	L_*^{-1}H_*=\frac{1}{d_*}M_*.
	\]
	Therefore
	\[
	S_*^\Omega L_*
	=
	S_*+\frac{1}{d_*}M_*,
	\]
	which yields \eqref{eq:S-Omega-reconstruction}.
\end{proof}

Combining the preceding results gives the local reconstruction chain
\begin{equation}
\boxed{
	DF(c_*)
	\longrightarrow
	d_*
	\longrightarrow
	H_*
	\longrightarrow
	S_*^\Omega ,
}
\label{eq:local-reconstruction-chain}
\end{equation}
provided \(Q_*\neq0\).

\subsection{The commuting regime}
\label{subsec:commuting-regime}

The complementary case has a different inverse character.

\begin{proposition}[Characterization of the exceptional regime]
	\label{prop:commuting-characterization}
	For exact return data,
	\[
	Q_*=0
	\]
	if and only if
	\begin{equation}
	M_*=M_*^*
	\qquad\text{and}\qquad
	[S_*,M_*]=0.
	\label{eq:commuting-characterization}
	\end{equation}
	Equivalently,
	\[
	[S_*,H_*]=0.
	\]
\end{proposition}

\begin{proof}
	If \(Q_*=0\), then
	\eqref{eq:skew-Q} gives
	\[
	M_*-M_*^*=0.
	\]
	Thus \(M_*=M_*^*\), and
	\[
	Q_*
	=
	S_*M_*-M_*S_*
	=
	[S_*,M_*],
	\]
	so \([S_*,M_*]=0\).
	
	Conversely, if \(M_*=M_*^*\) and
	\([S_*,M_*]=0\), then \(Q_*=0\).
	
	Finally, since
	\[
	H_*
	=
	\frac{1}{d_*}(I-d_*S_*)M_*,
	\]
	and \(I-d_*S_*\) is a function of \(S_*\), commutation of \(S_*\)
	with \(M_*\) is equivalent to commutation with \(H_*\).
\end{proof}

In this regime the linearization alone does not determine \(d_*\).
Indeed, for any trial value \(r>0\), the operator
\begin{equation}
H(r)
:=
\frac{1}{r}(I-rS_*)M_*
\label{eq:commuting-family-H}
\end{equation}
is self-adjoint whenever
\eqref{eq:commuting-characterization} holds, and satisfies
\[
M_*
=
r(I-rS_*)^{-1}H(r).
\]

Thus the factorization of the observed linearization admits a
one-parameter algebraic family. This does not assert that every
\(r>0\) is realized by an admissible global domain; it shows only that
\(DF(c_*)\) alone cannot distinguish these local factorizations.

\begin{corollary}[Umbilic cores]
	\label{cor:umbilic-local-ambiguity}
	If
	\[
	S_*=-\kappa_* I,
	\]
	then
	\[
	Q_*=0.
	\]
	Hence \(d_*\) cannot be recovered from \(DF(c_*)\) by the
	noncommuting reconstruction formula.
\end{corollary}

\begin{proof}
	Since \(S_*\) is scalar,
	\[
	Q_*
	=
	-\kappa_*(M_*-M_*^*).
	\]
	Equation~\eqref{eq:skew-Q} therefore gives
	\[
	(1+d_*\kappa_*)(M_*-M_*^*)=0.
	\]
	Because \(1+d_*\kappa_*>0\),
	\[
	M_*=M_*^*,
	\]
	and hence \(Q_*=0\).
\end{proof}

In particular, the spherical core belongs to the exceptional regime.

\begin{corollary}[One-dimensional boundary]
	\label{cor:one-dimensional-local-ambiguity}
	If
	\[
	\dim M=1,
	\]
	then \(Q_*=0\) at every fixed point. Consequently, the thickness
	value cannot be reconstructed from the derivative of the return map
	alone.
\end{corollary}

The noncommuting reconstruction mechanism is therefore intrinsically
multidimensional.

\subsection{Local consistency of observed linearizations}
\label{subsec:linearization-consistency}

The structured identity also yields a necessary condition for an
observed linearization to arise from the return geometry.

Let \(S_*\) be known and let \(M_*\) be an observed operator. A
necessary condition for local realizability is the existence of a
scalar \(d_*>0\) such that
\begin{equation}
M_*-M_*^*
=
d_*
\bigl(S_*M_*-M_*^*S_*\bigr).
\label{eq:linearized-compatibility}
\end{equation}

In the noncommuting case \(Q_*\neq0\), the only possible value is
\eqref{eq:d-star-reconstruction}. Thus exact local realizability
requires
\begin{equation}
\boxed{
	M_*-M_*^*
	-
	d_*Q_*
	=
	0.
}
\label{eq:local-realizability-residual}
\end{equation}

Once this condition holds, the reconstructed operator
\[
H_*=
\frac1{d_*}(I-d_*S_*)M_*
\]
is automatically self-adjoint. Further requirements---positivity of
\(d_*\), nondegeneracy when a Morse point is assumed, and
compatibility with an admissible outer boundary---remain geometric
constraints rather than consequences of the linearization alone.

\subsection{Local identifiability theorem}
\label{subsec:local-identifiability-summary}

The preceding results can be summarized as follows.

\begin{theorem}[Local inverse classification]
	\label{thm:local-inverse-classification}
	Let \(c_*\) be a fixed point of a return map generated by an
	admissible thickness function, and assume that the convex core \(C\)
	is known.
	
	Set
	\[
	M_*=DF(c_*)-I,
	\qquad
	Q_*=S_*M_*-M_*^*S_*.
	\]
	
	\begin{enumerate}
		\item
		If \(Q_*\neq0\), then
		\[
		d_*,
		\qquad
		H_*,
		\qquad
		S_*^\Omega
		\]
		are uniquely determined by \(DF(c_*)\) and \(S_*\) through
		\eqref{eq:d-star-reconstruction},
		\eqref{eq:H-star-reconstruction}, and
		\eqref{eq:S-Omega-reconstruction}.
		
		\item
		If \(Q_*=0\), then \(M_*\) is self-adjoint and commutes with \(S_*\).
		The derivative \(DF(c_*)\) does not determine \(d_*\) through the
		linearized factorization; instead it admits the one-parameter family
		\eqref{eq:commuting-family-H} at the algebraic level.
	\end{enumerate}
\end{theorem}
\section{Observable Dynamical and Topological Invariants}
\label{sec:observable-invariants}

Several dynamical and topological quantities associated with the
thickness landscape can be recovered directly from the observed return
map, without reconstructing the thickness value \(d_*\).  We record
here only the consequences needed for the inverse problem.

\subsection{Morse index from the observed linearization}
\label{subsec:observable-morse-index}

Let \(c_*\in\operatorname{Fix}(F)\) be nondegenerate and set
\[
M_*:=DF(c_*)-I.
\]
By \eqref{eq:BH-similarity}, \(M_*=B_*H_*\) is similar to the
self-adjoint operator
\[
K_*=B_*^{1/2}H_*B_*^{1/2},
\]
and \(K_*\) has the same inertia as \(H_*\). Hence
\begin{equation}
\boxed{
	\operatorname{ind}_M(c_*)
	=
	n_-(M_*).
}
\label{eq:observable-morse-index}
\end{equation}

\begin{proposition}[Observable Morse index]
	\label{prop:observable-morse-index}
	At every nondegenerate fixed point, the Morse index of the thickness
	function is determined by \(DF(c_*)\).
\end{proposition}

\begin{proof}
	Similarity of \(M_*\) and \(K_*\), together with
	\(\operatorname{inertia}(K_*)=\operatorname{inertia}(H_*)\), gives
	\eqref{eq:observable-morse-index}.
\end{proof}

Thus the Morse index remains observable even in the commuting regime.
Moreover, an observed \(M_*\) with nonreal spectrum or which is not
diagonalizable over \(\mathbb R\) cannot arise at a nondegenerate
fixed point of the present return geometry.

\subsection{Stable and unstable dimensions under strict ascent}
\label{subsec:observable-stability}

Assume the strict-ascent regime of \cite{EMB2026Global}. If
\(\mu_j\) is an eigenvalue of \(M_*\), the corresponding multiplier of
\(DF(c_*)\) is
\[
\Lambda_j=1+\mu_j.
\]
For every negative eigenvalue,
\begin{equation}
-2<\mu_j<0,
\label{eq:negative-spectrum-ascent}
\end{equation}
whereas positive eigenvalues satisfy \(\mu_j>0\). Consequently,
\begin{equation}
\boxed{
	\dim E^s(c_*)
	=
	\operatorname{ind}_M(c_*),
	\qquad
	\dim E^u(c_*)
	=
	m-\operatorname{ind}_M(c_*).
}
\label{eq:observable-stable-unstable}
\end{equation}
Thus the stable dimension agrees with the observable Morse index under
strict ascent; equation~\eqref{eq:observable-morse-index} itself does
not require that hypothesis.

\subsection{Fixed-point indices and degree}
\label{subsec:observable-fixed-index}

At a nondegenerate fixed point,
\begin{equation}
\operatorname{ind}_{\mathrm{fix}}(F,c_*)
=
\operatorname{sgn}\det(-M_*),
\label{eq:fixed-index-determinant}
\end{equation}
and therefore, by \eqref{eq:observable-morse-index},
\begin{equation}
\boxed{
	\operatorname{ind}_{\mathrm{fix}}(F,c_*)
	=
	(-1)^{m+\operatorname{ind}_M(c_*)}.
}
\label{eq:fixed-index-morse}
\end{equation}
Under strict ascent this becomes
\begin{equation}
\operatorname{ind}_{\mathrm{fix}}(F,c_*)
=
(-1)^{\dim E^u(c_*)}.
\label{eq:fixed-index-unstable}
\end{equation}

If \(d\) is Morse, the global fixed-point relation established in
\cite{EMB2026Morse} gives
\begin{equation}
\boxed{
	\sum_{c_*\in\operatorname{Fix}(F)}
	\operatorname{ind}_{\mathrm{fix}}(F,c_*)
	=
	\chi(\mathbb S^m)
	=
	1+(-1)^m.
}
\label{eq:global-fixed-index-balance}
\end{equation}
Since the Lefschetz number of
\(F:\mathbb S^m\to\mathbb S^m\) is
\[
L(F)=1+(-1)^m\deg(F),
\]
it follows that
\begin{equation}
\boxed{
	\deg(F)=1.
}
\label{eq:observable-degree-one}
\end{equation}
The degree-one condition does not require strict ascent.

\subsection{Morse counting constraints}
\label{subsec:observable-morse-counting}

For \(k=0,\ldots,m\), define
\begin{equation}
m_k(F)
:=
\#
\left\{
c_*\in\operatorname{Fix}(F):
n_-\bigl(DF(c_*)-I\bigr)=k
\right\},
\label{eq:observable-morse-counts}
\end{equation}
and
\begin{equation}
E_F(t)
:=
\sum_{k=0}^m m_k(F)t^k.
\label{eq:observable-equilibrium-polynomial}
\end{equation}
When \(d\) is Morse, the numbers \(m_k(F)\) are precisely its Morse
numbers. Hence, by the Morse inequalities on \(\mathbb S^m\),
\cite{Milnor1963,EMB2026Morse},
\begin{equation}
\boxed{
	E_F(t)
	=
	1+t^m+(1+t)Q(t),
	\qquad
	Q(t)\in\mathbb Z_{\ge0}[t].
}
\label{eq:observable-morse-polynomial}
\end{equation}
In particular,
\begin{equation}
\boxed{
	\#\operatorname{Fix}(F)=2+2Q(1),
}
\label{eq:observable-even-number}
\end{equation}
and
\begin{equation}
\boxed{
	\sum_{k=0}^m(-1)^k m_k(F)
	=
	1+(-1)^m.
}
\label{eq:observable-euler-balance}
\end{equation}

\subsection{Consequences for inverse data}
\label{subsec:inverse-consistency-summary}

For a candidate return map generated by a Morse thickness landscape,
the preceding identities provide directly testable necessary
conditions: at each fixed point \(DF(c_*)-I\) must have real nonzero
spectrum and be diagonalizable over \(\mathbb R\); the observable
indices must satisfy
\eqref{eq:observable-morse-polynomial} and
\eqref{eq:global-fixed-index-balance}; and the global map must satisfy
\[
\deg(F)=1.
\]
These constraints can be tested before reconstructing \(d\).
\section{Propagation and Global Reconstruction}
\label{sec:global-reconstruction}

The local reconstruction of Section~\ref{sec:local-inverse} provides
a thickness value at a noncommuting fixed point. We now show that one
such value determines the thickness function globally through the
inverse system of Section~\ref{sec:exact-inverse}.

\subsection{Transport across the fixed-point set}
\label{subsec:global-transport}

The inverse equation
\begin{equation}
(d-q_F)\nabla_M d=(I-dS)u_F
\label{eq:global-inverse-equation}
\end{equation}
was derived on \(M\setminus\operatorname{Fix}(F)\), where
\(u_F\neq0\).  The equation itself, however, remains valid at a fixed
point. Indeed, if \(c_*\in\operatorname{Fix}(F)\), then
\[
u_F(c_*)=0,
\qquad
q_F(c_*)=0,
\qquad
\nabla_Md(c_*)=0.
\]

Let
\[
\gamma:[0,T]\longrightarrow M
\]
be a piecewise \(C^1\) path and set
\[
y(t):=d(\gamma(t)).
\]
Then \eqref{eq:global-inverse-equation} gives
\begin{equation}
y'(t)
=
\frac{
	\left\langle
	(I-y(t)S_{\gamma(t)})u_F(\gamma(t)),
	\dot\gamma(t)
	\right\rangle
}{
	y(t)-q_F(\gamma(t))
}.
\label{eq:global-path-ode}
\end{equation}

For an admissible thickness function,
\begin{equation}
d-q_F=\sigma>0
\label{eq:global-sigma-positive}
\end{equation}
on all of \(M\). Hence the right-hand side of
\eqref{eq:global-path-ode} remains regular when a path passes through
a fixed point.

\subsection{Global uniqueness from one anchor value}
\label{subsec:global-anchor-uniqueness}

\begin{theorem}[Global one-anchor uniqueness]
	\label{thm:global-one-anchor-uniqueness}
	Let \(M=\partial C\) be connected and let
	\[
	d_1,d_2\in C^1(M)
	\]
	be two admissible thickness functions generating the same return map
	\(F\). If
	\[
	d_1(c_0)=d_2(c_0)
	\]
	at one point \(c_0\in M\), then
	\[
	\boxed{
		d_1=d_2
		\quad\text{on }M.
	}
	\]
\end{theorem}

\begin{proof}
	Both thickness functions satisfy
	\eqref{eq:global-inverse-equation}, with the same observable fields
	\(u_F\) and \(q_F\).
	
	Fix \(c\in M\). Since \(M\) is connected, choose a piecewise \(C^1\)
	path
	\[
	\gamma:[0,1]\to M,
	\qquad
	\gamma(0)=c_0,
	\qquad
	\gamma(1)=c.
	\]
	Set
	\[
	y_j(t):=d_j(\gamma(t)),
	\qquad
	j=1,2.
	\]
	Both functions solve the same scalar equation
	\eqref{eq:global-path-ode} and have the same initial value.
	
	For each admissible solution,
	\[
	y_j(t)-q_F(\gamma(t))>0.
	\]
	The right-hand side of \eqref{eq:global-path-ode} is therefore locally
	Lipschitz in \(y\) along the path. Uniqueness for the scalar ODE gives
	\[
	y_1(t)=y_2(t)
	\]
	for \(0\le t\le1\). Hence
	\[
	d_1(c)=d_2(c).
	\]
	Since \(c\) was arbitrary, \(d_1=d_2\) on \(M\).
\end{proof}

Thus exact non-uniqueness, if present, is reduced to the value of the
thickness at a single point.

\begin{corollary}[Domain reconstruction from one anchor]
	\label{cor:domain-from-anchor}
	If \(d(c_0)\) is known at one point, then the exact return map
	determines \(d\) uniquely on \(M\). Consequently,
	\[
	\partial\Omega
	=
	\{c+d(c)\nu(c):c\in M\}
	\]
	is uniquely determined.
\end{corollary}

\subsection{Global identifiability from a noncommuting fixed point}
\label{subsec:global-noncommuting-identifiability}

The anchor value required in
Theorem~\ref{thm:global-one-anchor-uniqueness} can be supplied by the
linearized return map.

\begin{theorem}[Global exact identifiability]
	\label{thm:global-exact-identifiability}
	Let \(F\) be the return map of an admissible thickness function on the
	known convex core \(C\). Assume that there exists a fixed point
	\(c_*\in\operatorname{Fix}(F)\) such that
	\begin{equation}
	Q_*
	=
	S_*M_*-M_*^*S_*
	\neq0,
	\qquad
	M_*=DF(c_*)-I.
	\label{eq:global-noncommuting-anchor}
	\end{equation}
	Then the thickness function is uniquely determined by \(F\) and
	\(C\). In particular, the outer boundary \(\partial\Omega\) is
	uniquely determined.
\end{theorem}

\begin{proof}
	By Theorem~\ref{thm:local-thickness-reconstruction}, the value
	\(d(c_*)\) is determined uniquely from \(DF(c_*)\) and \(S_*\):
	\[
	d(c_*)
	=
	\frac{
		\left\langle
		M_*-M_*^*,Q_*
		\right\rangle_{\mathrm{HS}}
	}{
		\|Q_*\|_{\mathrm{HS}}^2
	}.
	\]
	Thus any two admissible thickness functions generating the same
	return map have the same value at \(c_*\).
	Theorem~\ref{thm:global-one-anchor-uniqueness} then gives equality on
	all of \(M\).
	
	The normal representation
	\[
	\Phi(c)=c+d(c)\nu(c)
	\]
	determines \(\partial\Omega\).
\end{proof}

This result requires only one noncommuting fixed point. No
noncommutativity assumption is imposed elsewhere on \(M\).

\subsection{Constructive propagation}
\label{subsec:constructive-propagation}

The proof also gives a reconstruction procedure. Suppose a value
\[
d(c_0)=d_0
\]
is known, either from independent geometric information or from
Theorem~\ref{thm:local-thickness-reconstruction}.

For a path
\[
\gamma:[0,1]\to M,
\qquad
\gamma(0)=c_0,
\]
solve
\begin{equation}
\begin{cases}
\displaystyle
y'(t)
=
\frac{
	\left\langle
	(I-y(t)S_{\gamma(t)})u_F(\gamma(t)),
	\dot\gamma(t)
	\right\rangle
}{
	y(t)-q_F(\gamma(t))
},
\\[3mm]
y(0)=d_0.
\end{cases}
\label{eq:reconstruction-ivp}
\end{equation}
For exact return data,
\[
y(t)=d(\gamma(t)).
\]
In particular, if \(\gamma(1)=c\),
\begin{equation}
d(c)=y(1).
\label{eq:path-reconstructed-d}
\end{equation}

Thus the reconstruction consists of a scalar initial-value problem
along paths on the known manifold \(M\).

\subsection{Path independence and global compatibility}
\label{subsec:path-independence}

For an arbitrary candidate map \(F\), propagation by
\eqref{eq:reconstruction-ivp} need not be path independent.

Let \(\gamma\) be a path from \(c_0\) to \(c_1\), and denote by
\[
\mathcal T_\gamma(d_0)
\]
the terminal value of the solution of
\eqref{eq:reconstruction-ivp}, whenever that solution exists and
remains in the region
\[
y>q_F(\gamma(t)).
\]

If \(F\) is exact return data generated by a thickness function \(d\),
then
\begin{equation}
\boxed{
	\mathcal T_{\gamma_1}(d_0)
	=
	\mathcal T_{\gamma_2}(d_0)
}
\label{eq:path-independence}
\end{equation}
for any two paths \(\gamma_1,\gamma_2\) with the same endpoints and
with
\[
d_0=d(c_0).
\]

Equivalently, for every closed path \(\gamma\) based at \(c_0\),
\begin{equation}
\boxed{
	\mathcal T_\gamma(d_0)=d_0.
}
\label{eq:loop-consistency}
\end{equation}

These identities give global consistency conditions for candidate
return data.

\begin{proposition}[Path-independent reconstruction]
	\label{prop:path-independent-reconstruction}
	Fix \(c_0\in M\) and \(d_0>q_F(c_0)\). Assume that:
	
	\begin{enumerate}
		\item
		the initial-value problem
		\eqref{eq:reconstruction-ivp} is defined along every piecewise
		\(C^1\) path starting at \(c_0\);
		
		\item
		the terminal value depends only on the endpoint of the path;
		
		\item
		the resulting function
		\[
		d(c):=\mathcal T_\gamma(d_0)
		\]
		is of class \(C^1\).
	\end{enumerate}
	
	Then \(d\) satisfies the inverse equation
	\eqref{eq:global-inverse-equation}.
\end{proposition}

\begin{proof}
	Let \(c\in M\) and \(v\in T_cM\). Choose a path from \(c_0\) to \(c\),
	followed by a \(C^1\) curve \(\eta\) satisfying
	\[
	\eta(0)=c,
	\qquad
	\dot\eta(0)=v.
	\]
	Path independence identifies the transported value along \(\eta\)
	with \(d(\eta(t))\). Evaluating
	\eqref{eq:reconstruction-ivp} at \(t=0\) gives
	\[
	d_Md(c)[v]
	=
	\frac{
		\langle(I-d(c)S_c)u_F(c),v\rangle
	}{
		d(c)-q_F(c)
	}.
	\]
	Since \(v\) is arbitrary,
	\eqref{eq:global-inverse-equation} follows.
\end{proof}

Thus local integrability and global path independence are distinct
requirements: the former is differential, while the latter detects
global incompatibility.

\subsection{Conditional realizability of reconstructed data}
\label{subsec:conditional-realizability}

The inverse equation is not only a necessary differential condition.
Once a positive solution is available, it already reconstructs the
normal incidence relation underlying the return map.

\begin{proposition}[Conditional realizability]
	\label{prop:conditional-realizability}
	Let \(F:M\to M\) be a \(C^1\) candidate map, and let
	\(d:M\to(0,\infty)\) be a sufficiently regular solution of
	\eqref{eq:global-inverse-equation} satisfying
	\begin{equation}
	d(c)-q_F(c)>0,
	\qquad c\in M.
	\label{eq:realizability-positive-sigma}
	\end{equation}
	Set
	\[
	L_c:=I-d(c)S_c,
	\qquad
	w_c:=L_c^{-1}\nabla_M d(c),
	\]
	and define the normal graph
	\[
	\Phi_d(c):=c+d(c)\nu(c).
	\]
	Then, with
	\begin{equation}
	\sigma(c):=d(c)-q_F(c),
	\qquad
	t_d(c):=
	\sigma(c)\sqrt{1+|w_c|^2},
	\label{eq:realizability-return-distance}
	\end{equation}
	the inward unit normal to the reconstructed graph is
	\begin{equation}
	n_d(\Phi_d(c))
	=
	\frac{-\nu(c)+w_c}
	{\sqrt{1+|w_c|^2}},
	\label{eq:reconstructed-inward-normal}
	\end{equation}
	and one has the exact identity
	\begin{equation}
	\boxed{
		F(c)
		=
		\Phi_d(c)
		+
		t_d(c)\,n_d(\Phi_d(c)).
	}
	\label{eq:conditional-realizability-incidence}
	\end{equation}
	
	In particular, every solution of the inverse system reconstructs
	a normal graph whose inward normal ray passes exactly through the
	observed endpoint \(F(c)\).
	
	Assume in addition that \(\Phi_d(M)\) is an embedded outer boundary
	enclosing \(C\), and that
	\begin{equation}
	\boxed{
		\left\langle
		n_d(\Phi_d(c)),
		\nu(F(c))
		\right\rangle
		<0
	}
	\qquad
	(c\in M).
	\label{eq:first-hit-transversality}
	\end{equation}
	Then \(F(c)\) is the first intersection of the inward normal ray
	from \(\Phi_d(c)\) with \(M=\partial C\).
	
	Consequently, if the reconstructed outer boundary satisfies the
	remaining admissibility requirements defining the class
	\(\mathcal O_C\), then \(F\) is the return map of the reconstructed
	domain.
\end{proposition}

\begin{proof}
	By the inverse equation
	\eqref{eq:global-inverse-equation},
	\[
	\bigl(d-q_F\bigr)\nabla_M d
	=
	(I-dS)u_F.
	\]
	Hence, with
	\[
	\sigma=d-q_F,
	\qquad
	w=(I-dS)^{-1}\nabla_M d,
	\]
	we have
	\begin{equation}
	\sigma w=u_F.
	\label{eq:realizability-sigma-w}
	\end{equation}
	
	The normal formula
	\eqref{eq:reconstructed-inward-normal}
	is the normal-graph identity of
	Section~\ref{subsec:normal-graph-geometry}, applied to the
	reconstructed thickness \(d\).
	Using \eqref{eq:realizability-return-distance} and
	\eqref{eq:realizability-sigma-w},
	\begin{align*}
	\Phi_d(c)
	+t_d(c)n_d(\Phi_d(c))
	&=
	c+d\nu
	+
	\sigma(-\nu+w)
	\\
	&=
	c+(d-\sigma)\nu+\sigma w
	\\
	&=
	c+q_F\nu+u_F.
	\end{align*}
	By the definitions of \(u_F\) and \(q_F\),
	\[
	F(c)-c
	=
	u_F(c)+q_F(c)\nu(c),
	\]
	and therefore
	\eqref{eq:conditional-realizability-incidence} follows.
	
	It remains to verify the first-hit statement.
	Fix \(c\in M\) and set
	\[
	x:=F(c),
	\qquad
	n:=n_d(\Phi_d(c)),
	\qquad
	t_*:=t_d(c).
	\]
	Since \(C\) is convex, the tangent hyperplane to \(M\) at \(x\)
	is supporting, and hence
	\begin{equation}
	C
	\subset
	\left\{
	z\in\mathbb R^N:
	\langle z-x,\nu(x)\rangle\le0
	\right\}.
	\label{eq:supporting-halfspace-first-hit}
	\end{equation}
	For \(0\le t<t_*\), identity
	\eqref{eq:conditional-realizability-incidence} gives
	\[
	\Phi_d(c)+tn-x
	=
	(t-t_*)n.
	\]
	Therefore, by
	\eqref{eq:first-hit-transversality},
	\[
	\left\langle
	\Phi_d(c)+tn-x,
	\nu(x)
	\right\rangle
	=
	(t-t_*)\langle n,\nu(x)\rangle
	>0.
	\]
	Thus the open ray segment from \(\Phi_d(c)\) to \(x\) lies
	strictly outside the supporting half-space containing \(C\).
	It cannot intersect \(C\), and hence cannot meet \(M\), before
	reaching \(x=F(c)\).
	
	Thus \(F(c)\) is the first intersection of the reconstructed
	inward normal ray with \(M\). Under the stated admissibility
	assumptions, this is exactly the defining return construction.
\end{proof}

Proposition~\ref{prop:conditional-realizability} separates two parts
of the range problem.  The nonlinear inverse system already enforces
the exact normal-incidence relation.  What remains to be verified is
geometric: admissibility of the reconstructed outer boundary and the
first-hit property.  In particular, the inverse equation is stronger
than a merely necessary infinitesimal compatibility condition.
\subsection{Nonlinear holonomy and global anchor identifiability}
\label{subsec:nonlinear-holonomy}
To formulate the remaining global ambiguity in terms of a single
scalar datum, fix \(c_0\in M\) and define the admissible anchor set
\begin{equation}
\mathcal A_F(c_0)
:=
\left\{
a>q_F(c_0):
\begin{array}{l}
\text{there exists an admissible thickness function \(d\)}\\
\text{generating the return map \(F\) and satisfying \(d(c_0)=a\)}
\end{array}
\right\}.
\label{eq:admissible-anchor-set}
\end{equation}

By Theorem~\ref{thm:global-one-anchor-uniqueness}, each
\[
a\in\mathcal A_F(c_0)
\]
corresponds to at most one admissible thickness function generating
\(F\).  Thus the exact forward fiber is injectively parameterized by
the admissible anchor value.  The remaining identifiability question
is therefore whether the return data determine
\(\mathcal A_F(c_0)\) to contain at most one element.

The path transport of Section~\ref{subsec:path-independence} provides
a nonlinear mechanism for addressing this question.

Consider the open bundle
\begin{equation}
\mathcal E_F
:=
\left\{
(c,z)\in M\times\mathbb R:
z>q_F(c)
\right\}
\longrightarrow M
\label{eq:inverse-thickness-bundle}
\end{equation}
with projection
\[
\pi(c,z)=c.
\]
The one-form \(\omega_F\) defined in
\eqref{eq:inverse-one-form} induces on \(\mathcal E_F\) the one-form
\begin{equation}
\Theta_F
:=
dz-\omega_F(c,z).
\label{eq:inverse-connection-form}
\end{equation}
Since
\[
\Theta_F(\partial_z)=1,
\]
the distribution
\begin{equation}
\mathcal H_F
:=
\ker\Theta_F
\label{eq:inverse-horizontal-distribution}
\end{equation}
is complementary to the vertical direction and therefore defines an
 Ehresmann connection on \(\mathcal E_F\), in the standard sense
\cite{KobayashiNomizu1963}.

A lift
\[
t\longmapsto(\gamma(t),y(t))
\]
of a \(C^1\) path \(\gamma\) is horizontal precisely when
\[
\Theta_F(\dot\gamma(t),y'(t))=0,
\]
that is,
\[
y'(t)
=
\frac{
	\left\langle
	(I-y(t)S_{\gamma(t)})u_F(\gamma(t)),
	\dot\gamma(t)
	\right\rangle
}{
	y(t)-q_F(\gamma(t))
}.
\]
Thus the transport map \(\mathcal T_\gamma\) introduced in
Section~\ref{subsec:path-independence} is the parallel transport of
the connection \(\mathcal H_F\).

For a closed path \(\gamma\) based at \(c_0\), the corresponding map
\[
\mathcal T_\gamma:
\mathcal D_\gamma
\longrightarrow
\mathcal E_{F,c_0},
\]
where \(\mathcal D_\gamma\subset\mathcal E_{F,c_0}\) is the set of
initial values for which the horizontal lift exists along the whole
loop, will be called the nonlinear holonomy map associated with
\(\gamma\).

For exact return data generated by a thickness function \(d\),
the graph
\[
c\longmapsto(c,d(c))
\]
is a global horizontal section of \(\mathcal E_F\). Consequently, if
\[
a_*:=d(c_0),
\]
then
\begin{equation}
\boxed{
	\mathcal T_\gamma(a_*)=a_*
}
\label{eq:true-anchor-holonomy-fixed}
\end{equation}
for every closed path \(\gamma\) based at \(c_0\).

\begin{proposition}[Linearized loop holonomy]
	\label{prop:linearized-loop-holonomy}
	Let \(\gamma:[0,1]\to M\) be a \(C^1\) closed path based at \(c_0\),
	and let \(I\subset\mathcal D_\gamma\) be an open interval. For
	\(a\in I\), denote by \(y_a\) the horizontal lift satisfying
	\[
	y_a(0)=a.
	\]
	Then the holonomy map
	\[
	a\longmapsto\mathcal T_\gamma(a)=y_a(1)
	\]
	is \(C^1\), and
	\begin{equation}
	\boxed{
		\mathcal T_\gamma'(a)
		=
		\exp\left(
		-\int_0^1
		\frac{
			\left\langle
			(I-q_F(\gamma(t))S_{\gamma(t)})
			u_F(\gamma(t)),
			\dot\gamma(t)
			\right\rangle
		}{
			\bigl(y_a(t)-q_F(\gamma(t))\bigr)^2
		}
		\,dt
		\right).
	}
	\label{eq:loop-holonomy-derivative}
	\end{equation}
	In particular,
	\[
	\mathcal T_\gamma'(a)>0.
	\]
\end{proposition}

\begin{proof}
	Write
	\[
	G_\gamma(t,y)
	:=
	\frac{
		\left\langle
		(I-yS_{\gamma(t)})u_F(\gamma(t)),
		\dot\gamma(t)
		\right\rangle
	}{
		y-q_F(\gamma(t))
	}.
	\]
	Thus
	\[
	y_a'=G_\gamma(t,y_a).
	\]
	Set
	\[
	A(t)
	:=
	\langle u_F(\gamma(t)),\dot\gamma(t)\rangle,
	\qquad
	B(t)
	:=
	\langle
	S_{\gamma(t)}u_F(\gamma(t)),
	\dot\gamma(t)
	\rangle.
	\]
	Then
	\[
	G_\gamma(t,y)
	=
	\frac{A(t)-yB(t)}
	{y-q_F(\gamma(t))}.
	\]
	Differentiating with respect to \(y\) gives
	\begin{align}
	\partial_yG_\gamma(t,y)
	&=
	\frac{
		-B(t)(y-q_F)-(A(t)-yB(t))
	}{
		(y-q_F)^2
	}
	\nonumber\\
	&=
	-
	\frac{
		A(t)-q_F(\gamma(t))B(t)
	}{
		(y-q_F(\gamma(t)))^2
	}
	\nonumber\\
	&=
	-
	\frac{
		\left\langle
		(I-q_F(\gamma(t))S_{\gamma(t)})
		u_F(\gamma(t)),
		\dot\gamma(t)
		\right\rangle
	}{
		(y-q_F(\gamma(t)))^2
	}.
	\label{eq:transport-vertical-derivative}
	\end{align}
	
	Let
	\[
	Z_a(t)
	:=
	\frac{\partial y_a(t)}{\partial a}.
	\]
	Differentiability of solutions with respect to their initial value
	gives the variational equation
	\[
	Z_a'
	=
	\partial_yG_\gamma(t,y_a(t))Z_a,
	\qquad
	Z_a(0)=1.
	\]
	Therefore
	\[
	Z_a(1)
	=
	\exp\left(
	\int_0^1
	\partial_yG_\gamma(t,y_a(t))\,dt
	\right).
	\]
	Since
	\[
	Z_a(1)=\mathcal T_\gamma'(a),
	\]
	substitution of
	\eqref{eq:transport-vertical-derivative}
	proves \eqref{eq:loop-holonomy-derivative}.
\end{proof}

For an exact thickness function \(d\), define the one-form
\begin{equation}
\eta_d(c)[v]
:=
\partial_z\omega_F(c,d(c))[v]
=
-
\frac{
	\left\langle
	(I-q_F(c)S_c)u_F(c),v
	\right\rangle
}{
	(d(c)-q_F(c))^2
}.
\label{eq:linearized-holonomy-one-form}
\end{equation}
Since the horizontal lift starting from
\(a_*=d(c_0)\) is
\[
y_{a_*}(t)=d(\gamma(t)),
\]
Proposition~\ref{prop:linearized-loop-holonomy} yields
\begin{equation}
\boxed{
	\mathcal T_\gamma'(a_*)
	=
	\exp\left(
	\oint_\gamma\eta_d
	\right).
}
\label{eq:linearized-holonomy-multiplier}
\end{equation}

We call
\begin{equation}
\operatorname{Hol}^{\mathrm{lin}}_\gamma(d)
:=
\mathcal T_\gamma'(d(c_0))
\label{eq:linearized-holonomy-definition}
\end{equation}
the linearized holonomy multiplier of the inverse transport along
\(\gamma\).

\begin{theorem}[Local anchor identifiability from nonlinear holonomy]
	\label{thm:local-holonomy-identifiability}
	Let \(F\) be exact return data generated by an admissible thickness
	function \(d\), and fix \(c_0\in M\). Set
	\[
	a_*:=d(c_0).
	\]
	Assume that there exists a closed \(C^1\) path \(\gamma\) based at
	\(c_0\) for which
	\begin{equation}
	\operatorname{Hol}^{\mathrm{lin}}_\gamma(d)
	\neq1.
	\label{eq:nontrivial-linearized-holonomy}
	\end{equation}
	Then \(a_*\) is an isolated fixed point of the nonlinear holonomy map
	\(\mathcal T_\gamma\).
	
	Consequently, there exists a neighbourhood \(J\) of \(a_*\) such
	that
	\begin{equation}
	\boxed{
		\mathcal A_F(c_0)\cap J
		=
		\{a_*\}.
	}
	\label{eq:local-anchor-isolation}
	\end{equation}
	Thus the admissible thickness generating \(F\) is locally unique in
	the forward fiber with respect to the anchor value.
\end{theorem}

\begin{proof}
	Define
	\[
	H_\gamma(a)
	:=
	\mathcal T_\gamma(a)-a.
	\]
	By \eqref{eq:true-anchor-holonomy-fixed},
	\[
	H_\gamma(a_*)=0.
	\]
	Moreover,
	\[
	H_\gamma'(a_*)
	=
	\mathcal T_\gamma'(a_*)-1.
	\]
	Assumption
	\eqref{eq:nontrivial-linearized-holonomy}
	therefore gives
	\[
	H_\gamma'(a_*)\neq0.
	\]
	Hence \(a_*\) is an isolated zero of \(H_\gamma\), and therefore an
	isolated fixed point of \(\mathcal T_\gamma\).
	
	Every admissible thickness function generating the same return map
	\(F\) must satisfy the loop-consistency condition
	\eqref{eq:loop-consistency}. Thus every
	\(a\in\mathcal A_F(c_0)\) is a fixed point of
	\(\mathcal T_\gamma\). The isolation just proved yields
	\eqref{eq:local-anchor-isolation}.
	Finally, Theorem~\ref{thm:global-one-anchor-uniqueness} shows that
	each admissible anchor corresponds to at most one admissible
	thickness function.
\end{proof}
\begin{proposition}[Global loop criterion for anchor uniqueness]
	\label{prop:global-holonomy-anchor-criterion}
	Let \(\gamma\) be a closed path based at \(c_0\), and let
	\(I\subset\mathcal D_\gamma\) be a connected interval containing all
	admissible anchors under consideration. Suppose that either
	\begin{equation}
	\mathcal T_\gamma'(a)<1
	\qquad
	\text{for every }a\in I,
	\label{eq:holonomy-contractive}
	\end{equation}
	or
	\begin{equation}
	\mathcal T_\gamma'(a)>1
	\qquad
	\text{for every }a\in I.
	\label{eq:holonomy-expansive}
	\end{equation}
	Then \(\mathcal T_\gamma\) has at most one fixed point in \(I\).
	
	In particular, if
	\[
	\mathcal A_F(c_0)\subset I,
	\]
	then
	\begin{equation}
	\boxed{
		\#\mathcal A_F(c_0)\le1.
	}
	\label{eq:global-holonomy-anchor-uniqueness}
	\end{equation}
	If \(F\) is known to be realizable, the admissible anchor is therefore
	unique, and so is the corresponding admissible thickness function.
\end{proposition}

\begin{proof}
	Again set
	\[
	H_\gamma(a)=\mathcal T_\gamma(a)-a.
	\]
	Under \eqref{eq:holonomy-contractive},
	\[
	H_\gamma'(a)
	=
	\mathcal T_\gamma'(a)-1<0
	\]
	throughout \(I\), so \(H_\gamma\) is strictly decreasing and has at
	most one zero. Under
	\eqref{eq:holonomy-expansive},
	\(H_\gamma\) is strictly increasing, and the same conclusion follows.
	
	Every admissible anchor belongs to the fixed-point set of
	\(\mathcal T_\gamma\), by
	\eqref{eq:loop-consistency}. Hence
	\eqref{eq:global-holonomy-anchor-uniqueness} follows. The final
	statement follows from
	Theorem~\ref{thm:global-one-anchor-uniqueness}.
\end{proof}
It is useful to rewrite the preceding condition directly in terms of
the observed transport. Define
\begin{equation}
J_\gamma(a)
:=
\int_0^1
\frac{
	\left\langle
	(I-q_F(\gamma(t))S_{\gamma(t)})
	u_F(\gamma(t)),
	\dot\gamma(t)
	\right\rangle
}{
	\bigl(y_a(t)-q_F(\gamma(t))\bigr)^2
}
\,dt.
\label{eq:holonomy-integral}
\end{equation}
Then
\[
\mathcal T_\gamma'(a)=e^{-J_\gamma(a)}.
\]
Thus a fixed nonzero sign of \(J_\gamma\) on \(I\) is sufficient for
global uniqueness of the loop-compatible anchor on \(I\).

The two identifiability mechanisms developed above are therefore
distinct.  A noncommuting fixed point with \(Q_*\neq0\) determines the
anchor from the local linearization \(DF(c_*)\), whereas nontrivial
loop holonomy can isolate the anchor using the full nonlinear return
data along a closed path.  In particular, failure of the local
noncommuting criterion does not by itself imply failure of
identifiability.

\subsection{The all-commuting case}
\label{subsec:global-commuting-case}

Suppose that
\begin{equation}
Q_*=0
\qquad
\text{for every }c_*\in\operatorname{Fix}(F).
\label{eq:all-fixed-commuting}
\end{equation}
Then the local linearized reconstruction mechanism of
Section~\ref{sec:local-inverse} does not provide an anchor value.
At each fixed point, \(DF(c_*)\) determines only the one-parameter
algebraic family described in Section~\ref{subsec:commuting-regime}.

This failure of the fixed-point linearization does not imply failure
of identifiability for the full return map.  Recall the admissible
anchor set
\[
\mathcal A_F(c_0)
\]
introduced in Section~\ref{subsec:nonlinear-holonomy}.
By Theorem~\ref{thm:global-one-anchor-uniqueness}, each
\[
a\in\mathcal A_F(c_0)
\]
corresponds to at most one admissible thickness function generating
\(F\).  Hence exact non-uniqueness is entirely reduced to the possible
existence of more than one admissible anchor value.

Section~\ref{subsec:nonlinear-holonomy} gives a second, genuinely
nonlinear mechanism for resolving this ambiguity.  If there exists a
closed path \(\gamma\) based at \(c_0\) such that
\[
\operatorname{Hol}^{\mathrm{lin}}_\gamma(d)
\neq 1,
\]
then Theorem~\ref{thm:local-holonomy-identifiability} shows that the
true anchor \(d(c_0)\) is locally isolated in
\(\mathcal A_F(c_0)\).  More strongly, if for some closed path the
holonomy derivative satisfies
\[
\mathcal T_\gamma'(a)<1
\qquad\text{for all }a\in I,
\]
or
\[
\mathcal T_\gamma'(a)>1
\qquad\text{for all }a\in I,
\]
on a connected interval \(I\) containing all admissible anchors, then
Proposition~\ref{prop:global-holonomy-anchor-criterion} gives
\begin{equation}
\boxed{
	\#\mathcal A_F(c_0)\le 1.
}
\label{eq:all-commuting-holonomy-uniqueness}
\end{equation}

Thus the all-commuting condition
\eqref{eq:all-fixed-commuting} expresses failure only of the
\emph{local linearized} identifiability mechanism.  The full nonlinear
return map may still determine the thickness uniquely through its loop
transport.

The remaining open problem is therefore more specific: characterize
those all-commuting return maps for which the nonlinear holonomy
criteria of Section~\ref{subsec:nonlinear-holonomy} do not determine
the admissible anchor set.  In particular, one must understand whether
higher-order or genuinely global compatibility can still force
\[
\mathcal A_F(c_0)
\]
to be a singleton when the available linearized holonomy multipliers
are neutral or otherwise fail to yield a monotonicity criterion.
\begin{remark}
	\label{rem:global-ascent-not-required}
	The global uniqueness results of this section do not require the
	strict-ascent assumptions of \cite{EMB2026Global}. Those assumptions
	control the orbit structure of \(F\); the reconstruction above uses
	the exact geometric return identity.
\end{remark}
\section{Stability and Conditioning of the Inverse Problem}
\label{sec:stability-inverse}

We now consider perturbed return data. The local reconstruction of
Section~\ref{sec:local-inverse} already identifies the quantity
\[
Q_*
=
S_*M_*-M_*^*S_*
\]
as the obstruction to the commuting regime. We show that its norm
also controls the conditioning of the local inverse problem.

\subsection{Noisy linearized data}
\label{subsec:noisy-linearized-data}
The local perturbation analysis below treats the fixed point \(c_*\)
as known.  Errors arising from the localization of \(c_*\) from noisy
return data are not included.

Assume that the convex core is known and that the observed
linearization at a fixed point is
\begin{equation}
\widehat M_*
=
M_*+E_*,
\label{eq:noisy-M}
\end{equation}
where \(E_*\) represents measurement or differentiation error.

For an operator \(M\), define
\begin{align}
\mathcal A(M)
&:=
M-M^*,
\label{eq:A-map}
\\
\mathcal Q_{S_*}(M)
&:=
S_*M-M^*S_*.
\label{eq:Q-map}
\end{align}
For exact return data,
\begin{equation}
\mathcal A(M_*)
=
d_*\mathcal Q_{S_*}(M_*).
\label{eq:exact-linearized-relation}
\end{equation}

Set
\[
\widehat A_*
:=
\mathcal A(\widehat M_*),
\qquad
\widehat Q_*
:=
\mathcal Q_{S_*}(\widehat M_*).
\]
When \(\widehat Q_*\neq0\), the least-squares estimate of the
thickness value is
\begin{equation}
\boxed{
	\widehat d_*
	=
	\frac{
		\langle \widehat A_*,\widehat Q_*\rangle_{\mathrm{HS}}
	}{
		\|\widehat Q_*\|_{\mathrm{HS}}^2
	}.
}
\label{eq:noisy-d-estimator}
\end{equation}
Equivalently, \(\widehat d_*\) minimizes
\begin{equation}
r\longmapsto
\|\widehat A_*-r\widehat Q_*\|_{\mathrm{HS}}^2.
\label{eq:least-squares-d}
\end{equation}
If the data are sufficiently close to exact noncommuting data, then
\(\widehat d_*>0\).

Write
\[
\delta A_*
=
E_*-E_*^*,
\qquad
\delta Q_*
=
S_*E_*-E_*^*S_*.
\]
Then
\begin{equation}
\widehat A_*
=
d_*Q_*+\delta A_*,
\qquad
\widehat Q_*
=
Q_*+\delta Q_*.
\label{eq:AQ-perturbations}
\end{equation}

\subsection{Conditioning of the thickness reconstruction}
\label{subsec:conditioning-thickness}

\begin{proposition}[Perturbation of the reconstructed thickness]
	\label{prop:thickness-perturbation}
	Assume \(Q_*\neq0\) and
	\begin{equation}
	2\|S_*\|_{\mathrm{op}}
	\|E_*\|_{\mathrm{HS}}
	<
	\|Q_*\|_{\mathrm{HS}}.
	\label{eq:small-noise-Q}
	\end{equation}
	Then \(\widehat Q_*\neq0\), and
	\begin{equation}
	\boxed{
		|\widehat d_*-d_*|
		\le
		\frac{
			2\bigl(1+d_*\|S_*\|_{\mathrm{op}}\bigr)
			\|E_*\|_{\mathrm{HS}}
		}{
			\|Q_*\|_{\mathrm{HS}}
			-
			2\|S_*\|_{\mathrm{op}}\|E_*\|_{\mathrm{HS}}
		}.
	}
	\label{eq:d-error-bound}
	\end{equation}
\end{proposition}

\begin{proof}
	Since
	\[
	\widehat A_*
	=
	d_*\widehat Q_*
	+
	\delta A_*
	-
	d_*\delta Q_*,
	\]
	equation~\eqref{eq:noisy-d-estimator} gives
	\[
	\widehat d_*-d_*
	=
	\frac{
		\langle
		\delta A_*-d_*\delta Q_*,
		\widehat Q_*
		\rangle_{\mathrm{HS}}
	}{
		\|\widehat Q_*\|_{\mathrm{HS}}^2
	}.
	\]
	Therefore
	\begin{equation}
	|\widehat d_*-d_*|
	\le
	\frac{
		\|\delta A_*-d_*\delta Q_*\|_{\mathrm{HS}}
	}{
		\|\widehat Q_*\|_{\mathrm{HS}}
	}.
	\label{eq:d-error-first-bound}
	\end{equation}
	Moreover,
	\[
	\|\delta A_*\|_{\mathrm{HS}}
	\le
	2\|E_*\|_{\mathrm{HS}},
	\]
	and
	\[
	\|\delta Q_*\|_{\mathrm{HS}}
	\le
	2\|S_*\|_{\mathrm{op}}\|E_*\|_{\mathrm{HS}}.
	\]
	Hence
	\[
	\|\delta A_*-d_*\delta Q_*\|_{\mathrm{HS}}
	\le
	2\bigl(1+d_*\|S_*\|_{\mathrm{op}}\bigr)
	\|E_*\|_{\mathrm{HS}}.
	\]
	Finally,
	\[
	\|\widehat Q_*\|_{\mathrm{HS}}
	\ge
	\|Q_*\|_{\mathrm{HS}}
	-
	2\|S_*\|_{\mathrm{op}}\|E_*\|_{\mathrm{HS}},
	\]
	which is positive by
	\eqref{eq:small-noise-Q}. Substitution into
	\eqref{eq:d-error-first-bound} proves the result.
\end{proof}

The estimate shows that the stability constant furnished by
Proposition~\ref{prop:thickness-perturbation} deteriorates as
\[
\|Q_*\|_{\mathrm{HS}}\longrightarrow0.
\]
This motivates the local sensitivity scale
\begin{equation}
\chi_*
:=
\frac{
	1+d_*\|S_*\|_{\mathrm{op}}
}{
	\|Q_*\|_{\mathrm{HS}}
}.
\label{eq:local-sensitivity-index}
\end{equation}
The estimate alone does not imply that this loss of conditioning is
sharp.  In Section~\ref{subsec:example-near-commuting} we exhibit an
explicit perturbation family for which the reconstruction error grows
like \(1/\|Q_*\|_{\mathrm{HS}}\), showing that the singular behavior is
genuine.

\subsection{Propagation to the Hessian and outer curvature}
\label{subsec:error-hessian-curvature}

Let \(\widehat d_*>0\) be reconstructed from
\(\widehat M_*\). Define
\begin{equation}
\widehat H_*
=
\frac{1}{\widehat d_*}
(I-\widehat d_*S_*)\widehat M_*.
\label{eq:noisy-H}
\end{equation}
Since
\[
H_*
=
\frac1{d_*}M_*-S_*M_*,
\]
we have
\begin{equation}
\widehat H_*-H_*
=
\left(
\frac1{\widehat d_*}
-
\frac1{d_*}
\right)M_*
+
\left(
\frac1{\widehat d_*}I-S_*
\right)E_*.
\label{eq:H-error-identity}
\end{equation}
Consequently,
\begin{equation}
\boxed{
	\|\widehat H_*-H_*\|_{\mathrm{HS}}
	\le
	\frac{
		|\widehat d_*-d_*|
	}{
		d_*\widehat d_*
	}
	\|M_*\|_{\mathrm{HS}}
	+
	\left(
	\frac1{\widehat d_*}
	+
	\|S_*\|_{\mathrm{op}}
	\right)
	\|E_*\|_{\mathrm{HS}}.
}
\label{eq:H-error-bound}
\end{equation}

Thus Hessian recovery is stable whenever the thickness reconstruction
is stable and the thickness remains bounded away from zero.

For the outer curvature, set
\begin{equation}
\widehat S_*^\Omega
=
\left(
S_*+\frac1{\widehat d_*}\widehat M_*
\right)
(I-\widehat d_*S_*)^{-1}.
\label{eq:noisy-S-Omega}
\end{equation}
Convexity gives
\[
\|(I-rS_*)^{-1}\|_{\mathrm{op}}\le1,
\qquad
r>0.
\]
Using the resolvent identity,
\begin{equation}
(I-\widehat d_*S_*)^{-1}
-
(I-d_*S_*)^{-1}
=
(\widehat d_*-d_*)
(I-\widehat d_*S_*)^{-1}
S_*
(I-d_*S_*)^{-1}.
\label{eq:resolvent-L}
\end{equation}
It follows that
\begin{align}
\|\widehat S_*^\Omega-S_*^\Omega\|_{\mathrm{HS}}
\le{}&
\frac{\|E_*\|_{\mathrm{HS}}}{\widehat d_*}
+
\frac{
	\|M_*\|_{\mathrm{HS}}
	|\widehat d_*-d_*|
}{
	d_*\widehat d_*
}
\nonumber\\
&+
\left\|
S_*+\frac1{d_*}M_*
\right\|_{\mathrm{HS}}
\|S_*\|_{\mathrm{op}}
|\widehat d_*-d_*|.
\label{eq:S-Omega-error-bound}
\end{align}

Hence the conditioning of both second-order reconstructions inherits
the conditioning of \(d_*\).

\subsection{Stability of pathwise reconstruction}
\label{subsec:pathwise-stability}

We next consider the propagation step of
Section~\ref{subsec:constructive-propagation}.
Let
\[
\gamma:[0,T]\to M
\]
be fixed. For exact data, the reconstructed thickness along
\(\gamma\) solves
\begin{equation}
y'
=
G(t,y),
\qquad
G(t,y)
:=
\frac{
	\langle
	(I-yS_{\gamma(t)})u_F(\gamma(t)),
	\dot\gamma(t)
	\rangle
}{
	y-q_F(\gamma(t))
}.
\label{eq:exact-path-G}
\end{equation}

Suppose instead that the observed fields are
\[
\widehat u_F=u_F+\delta u,
\qquad
\widehat q_F=q_F+\delta q,
\]
and let
\[
\widehat y'
=
\widehat G(t,\widehat y)
\]
denote the corresponding reconstruction.

Assume that there exists a family of intervals
\[
I_t\subset\mathbb R,
\qquad
0\le t\le T,
\]
such that
\[
y(t),\widehat y(t)\in I_t
\]
and, for some \(\eta>0\),
\begin{equation}
z-q_F(\gamma(t))\ge\eta,
\qquad
z-\widehat q_F(\gamma(t))\ge\eta
\label{eq:common-denominator-margin}
\end{equation}
for every
\[
z\in I_t,
\qquad
0\le t\le T.
\]
Assume moreover that the intervals \(I_t\), the geometry, and
\(\dot\gamma\) remain uniformly bounded.

Then both \(G\) and \(\widehat G\) are uniformly Lipschitz in the
\(y\)-variable on this common domain.  Hence there exists \(L>0\)
such that
\begin{equation}
|\widehat G(t,z_1)-\widehat G(t,z_2)|
\le
L|z_1-z_2|
\label{eq:path-Lipschitz}
\end{equation}
for \(z_1,z_2\in I_t\).

Moreover, there exists \(C>0\), depending only on the same bounds and
on \(\eta\), such that
\begin{equation}
|\widehat G(t,z)-G(t,z)|
\le
C\left(
|\delta u(\gamma(t))|
+
|\delta q(\gamma(t))|
\right)
\label{eq:path-data-error}
\end{equation}
for every \(z\in I_t\).

Writing
\[
\widehat G(t,\widehat y)-G(t,y)
=
\bigl[
\widehat G(t,\widehat y)-\widehat G(t,y)
\bigr]
+
\bigl[
\widehat G(t,y)-G(t,y)
\bigr],
\]
we obtain
\[
\frac{d}{dt}|\widehat y-y|
\le
L|\widehat y-y|
+
C\left(
|\delta u(\gamma(t))|
+
|\delta q(\gamma(t))|
\right)
\]
in the usual integral sense. Therefore, if
\[
\widehat y(0)=y(0)+\delta d_0,
\]
Grönwall's inequality yields
\begin{equation}
\boxed{
	|\widehat y(t)-y(t)|
	\le
	e^{Lt}
	\left[
	|\delta d_0|
	+
	C\int_0^t
	\left(
	|\delta u(\gamma(s))|
	+
	|\delta q(\gamma(s))|
	\right)\,ds
	\right].
}
\label{eq:pathwise-error-bound}
\end{equation}

Thus pathwise reconstruction is stable as long as the exact and
perturbed transport equations remain uniformly separated from their
respective singular sets
\[
y=q_F,
\qquad
y=\widehat q_F
\]
on a common reconstruction domain.

If the initial value is obtained from a noncommuting fixed point, the
initial error \(\delta d_0\) is controlled by
Proposition~\ref{prop:thickness-perturbation}.  Poor conditioning at
the anchor therefore propagates through the global reconstruction.

\subsection{Residuals and consistency diagnostics}
\label{subsec:inverse-residuals}

Perturbed data need not satisfy the exact geometric identities.
The reconstruction therefore provides natural residuals.

At a fixed point, define the linearized residual
\begin{equation}
\mathcal R_{\mathrm{lin}}
:=
\left\|
\widehat A_*
-
\widehat d_*\widehat Q_*
\right\|_{\mathrm{HS}}.
\label{eq:linearized-residual}
\end{equation}
For exact noncommuting data,
\[
\mathcal R_{\mathrm{lin}}=0.
\]
A large residual indicates that the observed derivative is poorly
described by the structured return factorization.

For global propagation, let \(\gamma\) be a closed path based at
\(c_0\), and let
\[
\mathcal T_\gamma(d_0)
\]
denote the propagated terminal value introduced in
Section~\ref{subsec:path-independence}. Define
\begin{equation}
\mathcal R_\gamma(d_0)
:=
\left|
\mathcal T_\gamma(d_0)-d_0
\right|.
\label{eq:loop-residual}
\end{equation}
Exact return data satisfy
\[
\mathcal R_\gamma(d_0)=0
\]
for every closed path.

The local residual
\eqref{eq:linearized-residual} and the loop residual
\eqref{eq:loop-residual} test different aspects of the data:
the former tests the fixed-point factorization, while the latter tests
global path consistency.

\subsection{Near-commuting data}
\label{subsec:near-commuting-data}

When
\[
\|Q_*\|_{\mathrm{HS}}
\]
is comparable to the error level, formula
\eqref{eq:noisy-d-estimator} becomes unreliable. In that regime the
linearized data contain insufficient information to distinguish the
thickness value from nearby commuting factorizations.

Regularization may stabilize the numerical estimate, for example by
replacing \eqref{eq:least-squares-d} with a penalized scalar problem
of the type
\begin{equation}
\min_{r>0}
\left\{
\|\widehat A_*-r\widehat Q_*\|_{\mathrm{HS}}^2
+
\lambda |r-r_0|^2
\right\},
\qquad
\lambda>0,
\label{eq:regularized-d}
\end{equation}
when independent prior information \(r_0\) is available; see
\cite{Engl1996}. Such a procedure improves numerical stability but
does not restore information absent in the limit \(Q_*=0\).

The commuting regime is therefore not merely an algebraic exception.
It is the singular limit of the local inverse problem.

\section{Examples: Identifiability, Ambiguity, and Conditioning}
\label{sec:examples-inverse}

We illustrate the inverse classification, its conditioning, and the
global reconstruction mechanism through explicit local models and a
global numerical example.

\subsection{Umbilic core: the commuting regime}
\label{subsec:example-umbilic}

Assume that \(c_*\) is an umbilic point of the convex core,
\[
S_*=-\kappa I,
\qquad
\kappa\ge0.
\]
Then
\[
B_*
=
d_*(I-d_*S_*)^{-1}
=
\frac{d_*}{1+d_*\kappa}\,I,
\]
and therefore
\begin{equation}
M_*
=
\frac{d_*}{1+d_*\kappa}\,H_*.
\label{eq:umbilic-M}
\end{equation}
Since \(H_*\) is self-adjoint, so is \(M_*\), and
\[
Q_*
=
S_*M_*-M_*^*S_*
=
0.
\]

The thickness value cannot therefore be recovered from \(DF(c_*)\)
by Theorem~\ref{thm:local-thickness-reconstruction}.  More explicitly,
for any \(r>0\), define
\begin{equation}
H(r)
=
\frac{1+r\kappa}{r}\,M_*.
\label{eq:umbilic-family}
\end{equation}
Then \(H(r)\) is self-adjoint and
\[
M_*
=
r(I-rS_*)^{-1}H(r).
\]
Thus the same observed linearization admits the one-parameter family
\eqref{eq:umbilic-family} at the algebraic level.

In particular, every spherical core lies in the commuting regime.

\subsection{A noncommuting local model}
\label{subsec:example-noncommuting}

Consider the two-dimensional tangent-space model
\begin{equation}
S_*
=
\begin{pmatrix}
-1&0\\
0&-2
\end{pmatrix},
\qquad
d_*=1,
\qquad
H_*
=
\begin{pmatrix}
1&1\\
1&-1
\end{pmatrix}.
\label{eq:noncommuting-data}
\end{equation}

Since
\[
I-d_*S_*
=
\begin{pmatrix}
2&0\\
0&3
\end{pmatrix},
\]
we obtain
\[
B_*
=
\begin{pmatrix}
1/2&0\\
0&1/3
\end{pmatrix}
\]
and hence
\begin{equation}
M_*
=
B_*H_*
=
\begin{pmatrix}
1/2&1/2\\
1/3&-1/3
\end{pmatrix}.
\label{eq:noncommuting-M}
\end{equation}

The skew-adjoint part is
\[
M_*-M_*^*
=
\begin{pmatrix}
0&1/6\\
-1/6&0
\end{pmatrix}.
\]
Moreover,
\[
Q_*
=
S_*M_*-M_*^*S_*
=
\begin{pmatrix}
0&1/6\\
-1/6&0
\end{pmatrix}.
\]
Thus
\[
Q_*\neq0
\]
and
\[
\|Q_*\|_{\mathrm{HS}}^2
=
\frac1{18}.
\]
Formula~\eqref{eq:d-star-reconstruction} gives
\[
d_*
=
\frac{
	\langle M_*-M_*^*,Q_*\rangle_{\mathrm{HS}}
}{
	\|Q_*\|_{\mathrm{HS}}^2
}
=
1.
\]

The Hessian is then recovered from
\eqref{eq:H-star-reconstruction}:
\[
H_*
=
(I-S_*)M_*
=
\begin{pmatrix}
1&1\\
1&-1
\end{pmatrix},
\]
which reproduces the original matrix in
\eqref{eq:noncommuting-data}.

Finally,
\eqref{eq:S-Omega-reconstruction} gives
\begin{align}
S_*^\Omega
&=
(S_*+M_*)(I-S_*)^{-1}
\nonumber\\
&=
\begin{pmatrix}
-1/4&1/6\\
1/6&-7/9
\end{pmatrix}.
\label{eq:noncommuting-SOmega}
\end{align}
This operator is self-adjoint, with
\[
\operatorname{tr}S_*^\Omega=-\frac{37}{36}<0,
\qquad
\det S_*^\Omega=\frac16>0,
\]
and is therefore negative definite.

The example shows explicitly that noncommutativity does not create an
operator ambiguity.  Here it supplies the information needed to
recover
\[
d_*,
\qquad
H_*,
\qquad
S_*^\Omega
\]
from the observed derivative.

\subsection{Approach to the commuting regime}
\label{subsec:example-near-commuting}

Keep
\[
S_*
=
\begin{pmatrix}
-1&0\\
0&-2
\end{pmatrix},
\qquad
d_*=1,
\]
and consider
\begin{equation}
H_\varepsilon
=
\begin{pmatrix}
1&\varepsilon\\
\varepsilon&-1
\end{pmatrix},
\qquad
\varepsilon\in\mathbb R.
\label{eq:Hepsilon}
\end{equation}
Then
\begin{equation}
M_\varepsilon
=
B_*H_\varepsilon
=
\begin{pmatrix}
1/2&\varepsilon/2\\
\varepsilon/3&-1/3
\end{pmatrix}.
\label{eq:Mepsilon}
\end{equation}
A direct calculation gives
\begin{equation}
M_\varepsilon-M_\varepsilon^*
=
Q_\varepsilon
=
\begin{pmatrix}
0&\varepsilon/6\\
-\varepsilon/6&0
\end{pmatrix}.
\label{eq:Qepsilon}
\end{equation}
Hence
\begin{equation}
\|Q_\varepsilon\|_{\mathrm{HS}}
=
\frac{|\varepsilon|}{3\sqrt2}.
\label{eq:Qepsilon-norm}
\end{equation}

For every \(\varepsilon\neq0\), the exact reconstruction formula gives
\[
d_*=1.
\]
At \(\varepsilon=0\), however,
\[
Q_0=0,
\]
and the system enters the commuting regime.

Since
\[
\|S_*\|_{\mathrm{op}}=2,
\]
the sensitivity scale
\eqref{eq:local-sensitivity-index} is
\begin{equation}
\chi_\varepsilon
=
\frac{1+2}{\|Q_\varepsilon\|_{\mathrm{HS}}}
=
\frac{9\sqrt2}{|\varepsilon|}.
\label{eq:epsilon-sensitivity}
\end{equation}
Thus
\[
\chi_\varepsilon\longrightarrow\infty
\qquad
\text{as }\varepsilon\longrightarrow0.
\]

More explicitly, let
\[
\widehat M_\varepsilon
=
M_\varepsilon+E
\]
with \(E\) measured in Hilbert--Schmidt norm.  If
\begin{equation}
4\|E\|_{\mathrm{HS}}
<
\frac{|\varepsilon|}{3\sqrt2},
\label{eq:epsilon-noise-condition}
\end{equation}
Proposition~\ref{prop:thickness-perturbation} gives
\begin{equation}
\boxed{
	|\widehat d_*-1|
	\le
	\frac{
		6\|E\|_{\mathrm{HS}}
	}{
		|\varepsilon|/(3\sqrt2)
		-
		4\|E\|_{\mathrm{HS}}
	}.
}
\label{eq:epsilon-error-bound}
\end{equation}
The deterioration is not merely an artifact of the upper bound.
Set
\[
q_\varepsilon:=\frac{\varepsilon}{6},
\qquad
J:=
\begin{pmatrix}
0&1\\
-1&0
\end{pmatrix},
\]
so that
\[
M_\varepsilon-M_\varepsilon^*
=
Q_\varepsilon
=
q_\varepsilon J.
\]

Consider the symmetric perturbation
\begin{equation}
E_\delta
=
\begin{pmatrix}
0&\delta\\
\delta&0
\end{pmatrix}.
\label{eq:symmetric-conditioning-perturbation}
\end{equation}
Since \(E_\delta=E_\delta^*\),
\[
\widehat A_\varepsilon
=
\widehat M_\varepsilon-\widehat M_\varepsilon^*
=
q_\varepsilon J.
\]
On the other hand,
\[
S_*E_\delta-E_\delta S_*
=
\delta J,
\]
and therefore
\[
\widehat Q_\varepsilon
=
(q_\varepsilon+\delta)J.
\]

The least-squares estimator can consequently be evaluated exactly:
\begin{equation}
\widehat d_\varepsilon
=
\frac{q_\varepsilon}{q_\varepsilon+\delta}.
\label{eq:exact-conditioning-example}
\end{equation}
Hence
\begin{equation}
\widehat d_\varepsilon-1
=
-\frac{\delta}{q_\varepsilon+\delta}.
\label{eq:exact-conditioning-error}
\end{equation}

For fixed \(\varepsilon\neq0\) and
\(|\delta|\ll|\varepsilon|\),
\begin{equation}
|\widehat d_\varepsilon-1|
=
\frac{6|\delta|}{|\varepsilon|}
+
o(|\delta|).
\label{eq:conditioning-first-order}
\end{equation}
Since
\[
\|E_\delta\|_{\mathrm{HS}}
=
\sqrt2\,|\delta|,
\]
we obtain
\begin{equation}
|\widehat d_\varepsilon-1|
=
\frac{3\sqrt2}{|\varepsilon|}
\|E_\delta\|_{\mathrm{HS}}
+
o(\|E_\delta\|_{\mathrm{HS}}).
\label{eq:conditioning-sharp}
\end{equation}
Using
\[
\|Q_\varepsilon\|_{\mathrm{HS}}
=
\frac{|\varepsilon|}{3\sqrt2},
\]
this becomes
\begin{equation}
\boxed{
	|\widehat d_\varepsilon-1|
	=
	\frac{
		\|E_\delta\|_{\mathrm{HS}}
	}{
		\|Q_\varepsilon\|_{\mathrm{HS}}
	}
	+
	o(\|E_\delta\|_{\mathrm{HS}}).
}
\label{eq:conditioning-Q-sharp}
\end{equation}

Thus the \(1/\|Q_\varepsilon\|_{\mathrm{HS}}\) loss of conditioning
is genuine and is attained by an explicit perturbation direction.
The commuting configuration \(\varepsilon=0\) is therefore a true
singular limit of the local inverse problem.

\subsection{A computed one-parameter example}
\label{subsec:computed-example}

We now quantify the transition from the commuting regime to the
identifiable regime in the family introduced in
Section~\ref{subsec:example-near-commuting}.  Recall that
\[
S_*
=
\begin{pmatrix}
-1&0\\
0&-2
\end{pmatrix},
\qquad
d_*=1,
\qquad
H_\varepsilon
=
\begin{pmatrix}
1&\varepsilon\\
\varepsilon&-1
\end{pmatrix},
\]
which gives
\[
M_\varepsilon
=
\begin{pmatrix}
1/2&\varepsilon/2\\
\varepsilon/3&-1/3
\end{pmatrix},
\qquad
Q_\varepsilon
=
\begin{pmatrix}
0&\varepsilon/6\\
-\varepsilon/6&0
\end{pmatrix}.
\]
Hence
\[
\|Q_\varepsilon\|_{\mathrm{HS}}
=
\frac{|\varepsilon|}{3\sqrt2},
\qquad
\chi_\varepsilon
=
\frac{9\sqrt2}{|\varepsilon|}.
\]

Table~\ref{tab:computed-epsilon} shows representative values.  The
regime \(\varepsilon=0\) is exactly commuting and therefore locally
ambiguous, whereas every \(\varepsilon\neq0\) is locally identifiable.
However, the conditioning deteriorates rapidly as
\(|\varepsilon|\to0\).

\begin{table}[h]
	\centering
	\caption{Computed values for the family \(H_\varepsilon\).}
	\label{tab:computed-epsilon}
	\begin{tabular}{c c c c}
		\hline
		\(\varepsilon\) &
		\(\|Q_\varepsilon\|_{\mathrm{HS}}\) &
		\(\chi_\varepsilon\) &
		regime \\
		\hline
		1   & \(1/(3\sqrt2)\approx 0.2357\) & \(9\sqrt2\approx 12.73\) & identifiable \\
		0.5 & \(1/(6\sqrt2)\approx 0.1179\) & \(18\sqrt2\approx 25.46\) & identifiable \\
		0.2 & \(1/(15\sqrt2)\approx 0.0471\) & \(45\sqrt2\approx 63.64\) & identifiable, ill-conditioned \\
		0.1 & \(1/(30\sqrt2)\approx 0.0236\) & \(90\sqrt2\approx 127.28\) & strongly ill-conditioned \\
		0   & \(0\) & \(+\infty\) & commuting / ambiguous \\
		\hline
	\end{tabular}
\end{table}

Figure~\ref{fig:conditioning-epsilon} visualizes the transition.  The
left axis displays the size of the noncommuting observable
\(\|Q_\varepsilon\|_{\mathrm{HS}}\), while the right axis displays the
sensitivity index \(\chi_\varepsilon\).  The figure makes explicit
that identifiability improves with the noncommuting signal, whereas
conditioning degenerates in the limit \(\varepsilon\to0\).

\begin{figure}[H]
	\centering
	\begin{tikzpicture}
	\begin{axis}[
	width=0.92\linewidth,
	height=6.8cm,
	xmin=-1, xmax=1,
	ymin=0, ymax=0.26,
	axis x line*=bottom,
	axis y line*=left,
	xlabel={$\varepsilon$},
	ylabel={$\|Q_\varepsilon\|_{\mathrm{HS}}$},
	domain=-1:1,
	samples=200,
	legend style={at={(0.03,0.97)},anchor=north west,font=\small},
	ytick={0,0.05,0.10,0.15,0.20,0.25},
	xtick={-1,-0.5,0,0.5,1},
	clip=false
	]
	
	\addplot[thick] {abs(x)/(3*sqrt(2))};
	\addlegendentry{$\|Q_\varepsilon\|_{\mathrm{HS}}=\dfrac{|\varepsilon|}{3\sqrt2}$}
	
	\addplot[dashed] coordinates {(0,0) (0,0.26)};
	\node[above] at (axis cs:0,0.255) {\small commuting};
	
	\end{axis}
	
	\begin{axis}[
	width=0.92\linewidth,
	height=6.8cm,
	xmin=-1, xmax=1,
	ymin=0, ymax=140,
	axis x line=none,
	axis y line*=right,
	ylabel={$\chi_\varepsilon$},
	domain=-1:-0.06,
	samples=200,
	ytick={0,20,40,60,80,100,120,140},
	xtick=\empty,
	clip=false
	]
	
	\addplot[thick] {9*sqrt(2)/abs(x)};
	\addlegendentry{$\chi_\varepsilon=\dfrac{9\sqrt2}{|\varepsilon|}$}
	
	\end{axis}
	
	\begin{axis}[
	width=0.92\linewidth,
	height=6.8cm,
	xmin=-1, xmax=1,
	ymin=0, ymax=140,
	axis x line=none,
	axis y line=none,
	domain=0.06:1,
	samples=200,
	xtick=\empty,
	ytick=\empty,
	clip=false
	]
	
	\addplot[thick] {9*sqrt(2)/abs(x)};
	
	\end{axis}
	\end{tikzpicture}
	\caption{Computed transition from the commuting regime to the
		noncommuting identifiable regime in the family
		\eqref{eq:Hepsilon}.  The noncommuting observable
		\(\|Q_\varepsilon\|_{\mathrm{HS}}\) vanishes linearly at
		\(\varepsilon=0\), while the sensitivity index
		\(\chi_\varepsilon\) blows up like \(1/|\varepsilon|\).}
	\label{fig:conditioning-epsilon}
\end{figure}
\subsection{A global circular reconstruction and holonomy test}
\label{subsec:global-circle-example}

We conclude with a global example illustrating reconstruction from the
full return map in a setting where the fixed-point linearized mechanism
is necessarily commuting.

Let \(C\subset\mathbb R^2\) be the unit disk, with
\[
c(\theta)=(\cos\theta,\sin\theta),
\qquad
0\le\theta\le2\pi.
\]
Then
\[
M=\partial C=\mathbb S^1,
\qquad
S=-I
\]
on the tangent bundle.  Consider the positive thickness
\begin{equation}
d(\theta)
=
0.8
+
0.2\bigl(
\cos\theta+\sin\theta+\cos2\theta+\sin2\theta
\bigr).
\label{eq:global-circle-thickness}
\end{equation}
Numerically,
\[
0.3814\lesssim d(\theta)\lesssim1.3484.
\]

Writing
\[
\tau(\theta)=(-\sin\theta,\cos\theta),
\qquad
w(\theta)
=
\frac{d'(\theta)}{1+d(\theta)},
\]
the reconstructed inward normal is
\[
n_d(\Phi_d(\theta))
=
\frac{-c(\theta)+w(\theta)\tau(\theta)}
{\sqrt{1+w(\theta)^2}}.
\]
Solving the intersection equation with the unit circle and choosing
the smaller positive root corresponding to the first intersection gives
\begin{equation}
\sigma(\theta)
=
\frac{
	1+d(\theta)
	-
	\sqrt{
		1-w(\theta)^2\bigl((1+d(\theta))^2-1\bigr)
	}
}{
	1+w(\theta)^2
}.
\label{eq:global-circle-sigma}
\end{equation}
Hence the observable return components are
\begin{equation}
q_F(\theta)=d(\theta)-\sigma(\theta),
\qquad
u_F(\theta)
=
\sigma(\theta)w(\theta)\tau(\theta).
\label{eq:global-circle-data}
\end{equation}

For the thickness \eqref{eq:global-circle-thickness}, the radicand in
\eqref{eq:global-circle-sigma} remains uniformly positive; its
numerical minimum is approximately
\[
0.5037.
\]
Moreover,
\[
\max_{\theta}
\left\langle
n_d(\Phi_d(\theta)),
\nu(F(\theta))
\right\rangle
\approx -0.6510<0,
\]
so the first-hit transversality condition of
Proposition~\ref{prop:conditional-realizability} is satisfied with a
substantial margin.

Since \(S=-I\), propagation around the circle reduces to the scalar
initial-value problem
\begin{equation}
y'(\theta)
=
\frac{
	(1+y(\theta))\,u(\theta)
}{
	y(\theta)-q_F(\theta)
},
\qquad
u_F(\theta)=u(\theta)\tau(\theta).
\label{eq:global-circle-transport}
\end{equation}
Taking
\[
y(0)=d(0)=1.2
\]
recovers \eqref{eq:global-circle-thickness} over the entire circle to
numerical precision and gives
\[
\mathcal T_\gamma(1.2)=1.2,
\]
where \(\gamma\) is the positively oriented loop
\(\theta\mapsto c(\theta)\).

The linearized holonomy multiplier is obtained from
\eqref{eq:linearized-holonomy-multiplier}.  Numerically,
\[
\oint_\gamma\eta_d
\approx
0.02905,
\]
and therefore
\begin{equation}
\operatorname{Hol}^{\mathrm{lin}}_\gamma(d)
=
\mathcal T_\gamma'(1.2)
\approx
1.0295
\neq1.
\label{eq:global-circle-holonomy}
\end{equation}
Thus Theorem~\ref{thm:local-holonomy-identifiability} applies: the
true anchor \(d(0)=1.2\) is locally isolated by the nonlinear loop
transport.

For illustration, the numerically computed loop residual
\[
H_\gamma(a):=\mathcal T_\gamma(a)-a
\]
has the values
\[
\begin{array}{c|c|c}
a & \mathcal T_\gamma(a) & H_\gamma(a)\\
\hline
1.10 & 1.09606 & -3.94\times10^{-3}\\
1.15 & 1.14832 & -1.68\times10^{-3}\\
1.19 & 1.18970 & -3.02\times10^{-4}\\
1.20 & 1.20000 & \phantom{-}0\\
1.21 & 1.21029 & \phantom{-}2.88\times10^{-4}\\
1.25 & 1.25132 & \phantom{-}1.32\times10^{-3}\\
1.30 & 1.30239 & \phantom{-}2.39\times10^{-3}
\end{array}
\]
and therefore changes sign at the exact anchor.

Because \(\dim M=1\), Corollary~\ref{cor:one-dimensional-local-ambiguity}
implies
\[
Q_*=0
\]
at every fixed point.  The example therefore lies entirely in the
commuting regime from the viewpoint of fixed-point linearization, yet
the full nonlinear return data detect the anchor through loop
holonomy.  It provides a global numerical counterpart to the local
examples of Sections~\ref{subsec:example-umbilic}--%
\ref{subsec:computed-example}.
\section{Discussion and Open Problems}
\label{sec:discussion-open}

The inverse problem naturally separates into four questions:
reconstruction from the exact return displacement and its global
transport, reconstruction from fixed-point linearizations,
realizability of candidate return data, and recovery of observable
dynamical and topological invariants. The exact theory
leaves a small number of unresolved issues.

\subsection{Residual ambiguity in the all-commuting case}
\label{subsec:discussion-all-commuting}

The all-commuting condition
\[
Q_*=0
\qquad
\text{for every }c_*\in\operatorname{Fix}(F)
\]
does not imply non-identifiability.  As shown in
Section~\ref{subsec:nonlinear-holonomy}, nonlinear loop transport can
still isolate an admissible anchor locally, and under a suitable
monotonicity condition it can determine the anchor uniquely.

The unresolved problem is therefore the residual regime in which the
fixed-point linearizations are all commuting and the available loop
holonomies do not determine the anchor.  In particular, one must
understand configurations in which
\[
\operatorname{Hol}^{\mathrm{lin}}_\gamma(d)=1
\]
for all relevant closed paths, or more generally in which the
nonlinear maps \(\mathcal T_\gamma\) admit several compatible fixed
points.

The remaining question is:

\begin{quote}
	For exact return data in the all-commuting regime, what geometric
	conditions guarantee that the admissible anchor set
	\(\mathcal A_F(c_0)\) is a singleton, and under what conditions can
	more than one admissible anchor actually occur?
\end{quote}

Highly symmetric cores provide a natural test class.  For a spherical
core,
\[
S=-\kappa I,
\]
the fixed-point reconstruction is commuting at every critical point,
so any remaining identifiability mechanism must use information
beyond the fixed-point derivatives.

Even when all linearized holonomy multipliers are neutral, uniqueness
may still be enforced by higher-order properties of the nonlinear
holonomy maps, by the simultaneous fixed-point conditions for several
loops, or by the global admissibility conditions of the reconstructed
normal graph.  Whether these mechanisms eliminate the residual
ambiguity remains open.
\subsection{Realizability of candidate return maps}
\label{subsec:discussion-realizability}

The inverse problem also contains a range question: given a candidate
map
\[
F:M\to M,
\]
when does there exist an admissible domain
\(\Omega\in\mathcal O_C\) whose return map is \(F\)?

Several compatibility conditions are necessary.  If the endpoint map
is assumed a priori to take values in \(M\), the closure relation
\[
q_F(c)=h_c(u_F(c))
\]
is automatic.  The nontrivial conditions include local integrability and global
loop consistency of the nonlinear inverse system
\[
(d-q_F)\nabla_Md=(I-dS)u_F,
\]
together with the structured fixed-point identities of
Section~\ref{sec:local-inverse} and, in the Morse setting, the
topological constraints of
Section~\ref{sec:observable-invariants}.

Proposition~\ref{prop:conditional-realizability} gives a partial
converse.  If a positive solution \(d\) of the inverse system exists
with
\[
d-q_F>0,
\]
then the reconstructed normal graph
\[
\Phi_d(c)=c+d(c)\nu(c)
\]
already satisfies the exact normal-incidence identity
\[
F(c)
=
\Phi_d(c)
+
t_d(c)n_d(\Phi_d(c)).
\]
Thus the first-order inverse system recovers more than an
infinitesimal compatibility relation: it reconstructs the inward normal ray from the reconstructed outer graph passing exactly through the observed endpoint.

If, in addition, the reconstructed graph is an embedded outer boundary
enclosing \(C\) and satisfies the transversality condition
\[
\left\langle
n_d(\Phi_d(c)),\nu(F(c))
\right\rangle
<0,
\]
then \(F(c)\) is the first intersection of this inward normal ray with
\(\partial C\).  Hence, once the remaining defining admissibility
requirements of \(\mathcal O_C\) are satisfied, the reconstructed
domain has return map \(F\).

The unresolved realizability problem is therefore more specific than
mere solvability of the inverse equation.  A complete range
characterization would require conditions, preferably expressed
directly in terms of \(F\) and the known geometry of \(C\), that ensure
simultaneously:

\begin{enumerate}
	\item the existence of a global positive solution of the inverse
	system with \(d-q_F>0\);
	\item admissibility of the reconstructed normal graph as an outer
	boundary in \(\mathcal O_C\);
	\item the required first-hit geometry of the inward normal rays.
\end{enumerate}

Equivalently, one seeks an intrinsic characterization of the range of
the forward map
\[
\Omega\in\mathcal O_C
\longmapsto
F.
\]
Proposition~\ref{prop:conditional-realizability} reduces this range
problem to the combination of global inverse-system solvability and
geometric admissibility, but a complete criterion formulated solely
from observable return data remains open.
\subsection{Lower regularity}
\label{subsec:open-regularity}

The exact displacement identity originates in the geometric structure
of the return map and survives under weaker assumptions than those
used for the second-order reconstruction.  By contrast, the formulas
\[
DF(c_*)=I+B_*H_*
\]
and
\[
H_*=L_*S_*^\Omega L_*-L_*S_*
\]
require enough regularity to define the Hessian and the relevant shape
operators.

It is therefore natural to ask how much of the inverse theory survives
when \(d\) or \(\partial\Omega\) has only limited regularity.
Possible targets include:

\begin{enumerate}
	\item reconstruction of \(d\) from the exact inverse system when
	\(d\in C^{1,1}(M)\);
	
	\item almost-everywhere versions of the local inverse identities;
	
	\item weak formulations of the Pfaffian system
	\eqref{eq:inverse-pfaffian};
	
	\item stability of reconstruction under nonsmooth perturbations of the
	observed return map.
\end{enumerate}

This would connect the inverse theory more directly with the original
regularity framework of the class \(\mathcal O_C\)
\cite{Barkatou2002}.

\subsection{Degenerate critical sets and higher-order information}
\label{subsec:open-degenerate}

The local reconstruction developed in Section~\ref{sec:local-inverse}
uses the derivative \(DF(c_*)\).  At a degenerate critical point,
\[
\ker H_*\neq\{0\},
\]
and the linearization does not contain all local information about the
thickness landscape.

The Morse--Bott setting studied in
\cite{EMB2026Morse} already shows one structured form of degeneracy:
if \(\Sigma\subset\operatorname{Crit}(d)\) is a critical submanifold,
then
\[
\ker H_c=T_c\Sigma
\]
and
\[
DF(c)|_{T_c\Sigma}=I.
\]
The transverse spectrum remains observable, while tangential
directions carry multiplier \(1\).

For more general degeneracies, reconstruction should involve higher
jets of the return map.  A basic problem is to determine which part of
the jet
\[
J^kF(c_*)
\]
determines the corresponding jet
\[
J^k d(c_*),
\]
once the geometry of \(C\) is known.

This is a genuinely higher-order inverse problem and is not addressed
by the present linear theory.

\subsection{From local conditioning to global stability}
\label{subsec:open-global-stability}

Section~\ref{sec:stability-inverse} gives local perturbation estimates
at a noncommuting anchor and pathwise stability provided the exact and
perturbed transport equations remain uniformly separated from their
respective singular sets on a common reconstruction domain.

One natural target is an estimate of the form
\begin{equation}
\|\widehat d-d\|_{X}
\le
C\,\|\widehat F-F\|_{Y},
\label{eq:discussion-global-stability}
\end{equation}
for suitable function spaces \(X\) and \(Y\), with explicit dependence
of \(C\) on the noncommuting anchor and on the geometric bounds of the
core.

This does not imply instability of the full inverse problem in the
commuting regime.

The analysis of Section~\ref{sec:stability-inverse} shows that a
global estimate based on a noncommuting fixed-point anchor cannot
remain uniform as
\[
\|Q_*\|_{\mathrm{HS}}\to0.
\]
This does not imply instability of the full inverse problem in the
commuting regime.  Section~\ref{subsec:nonlinear-holonomy} shows that
the anchor may instead be determined by nonlinear loop transport.

For a holonomy-based reconstruction, a natural local nondegeneracy
gap is
\[
\delta_\gamma
:=
\left|
\mathcal T_\gamma'(d(c_0))-1
\right|
=
\left|
\operatorname{Hol}^{\mathrm{lin}}_\gamma(d)-1
\right|.
\]
Its vanishing corresponds precisely to the loss of the simple
fixed-point isolation mechanism of
Theorem~\ref{thm:local-holonomy-identifiability}.  A quantitative
stability theory should therefore determine how reconstruction
constants depend on \(\delta_\gamma\), just as the local
fixed-point reconstruction depends on \(\|Q_*\|_{\mathrm{HS}}\).

A complete global stability theory should distinguish the two
anchor-recovery mechanisms: noncommuting fixed-point reconstruction
and nonlinear holonomy reconstruction, and quantify the respective
degeneracies of each.

\end{document}